\documentclass[12pt,reqno]{amsart}

\usepackage{color}
\usepackage{amsmath, amsthm, amssymb}
\usepackage{amsfonts}
\usepackage[ansinew]{inputenc}
\usepackage[dvips]{epsfig}
\usepackage{graphicx}
\usepackage[english]{babel}
\usepackage{hyperref}

\usepackage{cite}
\usepackage{graphicx}
\usepackage{amscd}
\usepackage{color}
\usepackage{bm}
\usepackage{enumerate}

\usepackage{verbatim}
\usepackage{hyperref}
\usepackage{amstext}
\usepackage{latexsym}
\usepackage{mathrsfs}
\newcommand{\scrS }{\mathscr{S}}

\theoremstyle{plain}
\newtheorem{Theorem}{Theorem}[section]

\newtheorem{Proposition}[Theorem]{Proposition}
\newtheorem{Lemma}[Theorem]{Lemma}
\newtheorem{Remark}[Theorem]{Remark}

\numberwithin{Theorem}{section}
\numberwithin{equation}{section}

\def\proof{\noindent{{\bf Proof. }}}
\def\square{\vbox{
\hrule height .4pt \hbox{\vrule width .4pt height 7pt \kern 7pt
\vrule width .4pt} \hrule height .4pt }}
\def\QED{\hfill {$\square$}\goodbreak \medskip}

\newcommand{\average}{{\mathchoice {\kern1ex\vcenter{\hrule height.4pt
width 6pt depth0pt} \kern-9.7pt} {\kern1ex\vcenter{\hrule
height.4pt width 4.3pt depth0pt} \kern-7pt} {} {} }}

\def\R{\mathbb{R}}

\renewcommand{\a }{\alpha }

\renewcommand{\d}{\delta }

\renewcommand{\phi}{\varphi}

\newcommand{\e }{\varepsilon }
\newcommand{\g }{\gamma}

\renewcommand{\l }{\lambda }
\renewcommand{\L }{\Lambda }

\newcommand{\vp }{\varphi }

\newcommand{\s }{\sigma }

\renewcommand{\t }{\tau }

\renewcommand{\th }{\theta }

\renewcommand{\O }{\Omega }

\newcommand{\ov}{\overline}

\newcommand{\be}{\begin{equation}}
\newcommand{\ee}{\end{equation}}

\newcommand{\de}{\partial}

\newcommand{\ti}{\widetilde}

\newcommand{\N}{\mathbb{N}}

\newcommand{\cA}{{\mathcal A}}

\newcommand{\cD}{{\mathcal D}}

\newcommand{\cF}{{\mathcal F}}

\newcommand{\cH}{{\mathcal H}}

\newcommand{\cK}{{\mathcal K}}

\newcommand{\cO}{{\mathcal O}}

\newcommand{\cU}{{\mathcal U}}

\renewcommand{\epsilon}{\varepsilon}

\begin{document}

\title[A new  class of exceptional subdomains of $\R^3$] 
{On an electrostatic problem and a new  class of exceptional subdomains of $\R^3$}

\author{Mouhamed Moustapha Fall}
\address{M. M. F.: African Institute for Mathematical Sciences in Senegal, KM 2, Route de
Joal, B.P. 14 18. Mbour, Senegal.}
\email{mouhamed.m.fall@aims-senegal.org}


\author[I. A. Minlend]{Ignace Aristide Minlend} 
\address{Faculty of Economics and Applied Management, University of Douala}
\email{\small{ignace.minlend@univ-douala.com} }

\author{Tobias Weth}
\address{T.W.:  Goethe-Universit\"{a}t Frankfurt, Institut f\"{u}r Mathematik.
Robert-Mayer-Str. 10 D-60629 Frankfurt, Germany.}

\email{weth@math.uni-frankfurt.de}

\keywords{Overdetermined problems, exceptional domains, electrostatic equilibrium}


\begin{abstract}
We study  the existence of nontrivial  unbounded  surfaces $S\subset \R^3$ with the property that the constant charge distribution on $S$ is an electrostatic equilibrium, i.e. the resulting electrostatic force is normal to the surface at each point on $S$. Among bounded regular surfaces $S$, only the round sphere has this property by a result of Reichel \cite{reichel1} (see also Mendez and Reichel \cite{ReichelMendez}) confirming a conjecture of P. Gruber. In the present paper, we show the existence of nontrivial unbounded  exceptional domains $\O \subset \R^3$  whose boundaries  $S=\partial \O$ enjoy the above property.
\end{abstract}
\maketitle

\textbf{MSC 2010}:  35J57, 35J66,  35N25, 35J25, 35R35, 58J55 
\section{Introduction and main result}

A smooth domain $\O$ of the Euclidean space $\R^N$ is called  \emph{exceptional} if  there exists a positive harmonic function in $\Omega$ which satisfies the overdetermined boundary conditions $u=0$ and $\partial_\nu u = {\rm const\ne 0}$ on $\partial \Omega$, where $\partial_\nu$ denotes the outer normal derivative on $\partial \Omega$. Half spaces in $\R^N$, $N \ge 1$ and complements of balls in $\R^N$, $N \ge 2$ are trivial examples of exceptional domains. Moreover, if $\Omega \subset \R^N$ is exceptional, then $\Omega \times \R^k$ is obviously exceptional in $\R^{N+k}$. In particular, cylindrical domains of the type $B \times \R^k$ with a ball $B \subset \R^N$ are exceptional in $\R^{N+k}$ if $N \ge 2$.

The problem of finding and classifying \emph{exceptional domains} was first  studied by L. Hauswirth, F. H\'{e}lein, and F.
Pacard  in  the seminal paper \cite{hauswirth-et-al}, where the authors exhibit, in particular, the nontrivial example
\begin{equation}
  \label{eq:catenoid-ex}
\Omega_0:= \{(x,y) \in \R^2: |y  |<\frac{\pi}{2}+\cosh (x)\}  
\end{equation}
in the plane. The classification problem in the planar case was then further studied by Khavinson, Lundberg and Teodorescu \cite{KTeo}, who showed that within a specific class of planar expectional domains characterized by additional assumptions, the only examples up to rotation and translation are the exterior of
a disk, a halfplane and the nontrivial domain $\Omega_0$. Shortly later, Traizet \cite{Traizet} developed a different approach which allowed him to characterize these three examples as the only planar exceptional domains having finitely boundary components. Moreover, he established  a one-to-one correspondence between  exceptional domains  of $\R^2$  and the so called  \emph{minimal bigraphs}. With the help of this correspondence, he also reveals the existence of a nontrial periodic  exceptional domain corresponding  to Scherk's simply periodic minimal bigraphs, \cite[Example 7.3]{Traizet}. It is worth noting that the previous examples are not the only exceptional domains in the plane. Indeed, a family of infinitely connected exceptional domains was  already discovered   in fluid dynamics   by G. R. Baker, P. G. Saffman and J. S. Sheffield  \cite{Baker}, when  modelling  hollow vortex equilibria  with  an infinite periodic array of vertices, see also \cite{CrowdyGreen}.

Despite the significant efforts in the previous literature, the structure of the set of exceptional domains in dimensions $N \ge 3$ remains largely unknown. With the help of an earlier result from \cite{reichel1}, the exteriors of balls were classified in \cite[Theorem 7.1]{KTeo} as the only exceptional domains  in  $\R^N$ whose complements  are  bounded, connected and have $C^{2, \alpha}$ boundaries. In the planar case, this classification holds under much weaker regularity assumptions, see \cite{EKS} and \cite{ReichelMendez}. On the other hand, analogues of the domain $\Omega_0$ in (\ref{eq:catenoid-ex}) in higher dimensions have been detected recently in \cite{LiuWangWei}.

In this   work, we deal  with the  construction of a new class of exceptional subdomains $\Omega \subset \R^3$ with the property that $\partial \Omega$ does not have constant principal curvatures. Specifically, we wish to study domains of the form $\Omega = \R^3 \setminus \overline D$, where $D \subset \R^3$ is regular open set which is periodic in $x_3$-direction and bounded in directions $x_1, x_2$ such that the  overdetermined problem  
\begin{equation}\label{eq: Poisson-Dir-exceptional}
\left\{
\begin{aligned}
    -\Delta u&=0&&\qquad \textrm{in} \quad  \Omega\vspace{3mm}\\
u&=0,  \quad \frac{\partial u}{\partial \nu}=const\ne 0 &&\qquad\textrm{on }\quad \partial\Omega
  \end{aligned}
\right.
\end{equation}
is solvable by a function $u$ of constant sign. Here, as before, $\nu$  is the unit outward normal vector to  the boundary
$\partial \O$.

The domains we wish to construct are complements of perturbed cylinders of the form 
\begin{equation}
  \label{eq:def-o-phi}
\O_\phi:=\left\{ (z,t)\in
\mathbb{R}^{2}\times\mathbb{R}\,:\, |z|>\phi(t)  \right\}\subset\R^3,
\end{equation}
where $\phi: \R \to (0,\infty)$ is a $2 \pi$-periodic function of class $C^{2,\alpha}$ for some $\alpha>0$. The case $\phi \equiv \lambda$ corresponds to the complement of a straight cylinder $\Omega_\lambda$ of radius $\lambda>0$. In this case, the function 
$$
(z,t) \mapsto u_\lambda(z,t)= \log \frac{|z|}{\lambda}
$$ 
is -- up to a multiplicative constant -- the unique nontrivial solution of (\ref{eq: Poisson-Dir-exceptional}) which is periodic in $t$ and grows at most logarithmically in $|z|$.
Our aim is to study possible bifurcation of nontrivial exceptional domains from the branch $\Omega_\lambda$, $\lambda>0$.

The domains we construct solve an electrostatic problem, as they enjoy the property that the constant charge distribution on $\partial \O_\phi$ is an electrostatic equilibrium. The question of determining equilibrium distributions has been intensively investigated by several authors in the literature, \cite{Payne, Philippin, Martensen, reichel1, reichel2, ReichelMendez}. In \cite{reichel1,reichel2}, Reichel addressed a conjecture by P. Gruber which states that the  equilibrium distribution is constant on the boundary of a domain if and only if the domain is a ball, and he proved the conjecture within the class of $C^{2,\alpha}$-domains. Later in \cite{ReichelMendez}, 
Mendez and Reichel proved the validity of the conjecture  in dimension $N=2$ for the class of bounded Lipschitz domains and in dimension $N\geq 3$ for the class of bounded convex domains. 

In  our case,  we are  concerned with the  unbounded  exceptional domains of the form 
$\O_\phi$ given in (\ref{eq:def-o-phi}). To  set-up our problem, we  intend to make the following ansatz which is directly inspired by the electrostatic interpretation of the problem. We put $Z_\phi:= \partial \Omega_\phi$ and define 
$$
u_\phi: \R^3 \to \R, \qquad u_\phi(x):=  \frac{1}{4\pi} \int_{Z_{\phi}}\left( \frac{1}{|x-y|} -\frac{1}{|y|}\right)d\sigma(y), \qquad x = (z,t) \in \R^3. 
$$
Then $u_\phi$ represents the normalized (because $Z_\phi$ is unbounded) electrostatic potential induced by a globally constant charge of density $1$ on $Z_\phi$, see \cite{ReichelMendez}. It is well known that $u_\phi$ is continuous on $\R^3$ and $u_\phi \in C^1(\overline {\O_\phi})\cap C^1(\overline {D_\phi})$, where  
$$
D_\phi:=\R^N \setminus \overline {\O_\phi} = \left\{ (z,t)\in
\mathbb{R}^{2}\times\mathbb{R}\,:\, |z|<\phi(t)  \right\}\subset\R^3.
$$
Moreover, $u_\phi$ is harmonic in $\O_\phi \cup D_\phi$, and on $Z_\phi$, see \cite[Theorem 1.11]{G. Verchota} 
it satisfies the {\em jump condition} 
\begin{equation}
  \label{eq:jump-condition}
\frac{\partial u_\phi}{\partial \nu^\pm} = \mathcal{N}_\phi \mp \frac{1}{2}.
\end{equation}
Here $\nu^+$ resp. $\nu^-$ is the unit normal on $Z_\phi$ pointing inside $\Omega_\phi$, $D_\phi$, respectively and $\mathcal{N}_\phi: Z_\phi \to \R$ is given by 
\begin{equation}
  \label{eq:electric-field}
\mathcal{N}_\phi(x):=\frac{1}{4\pi} \int_{Z_\phi} \frac{(y-x) \cdot \nu^+(x)}{|x-y|^3} d \sigma(y).
\end{equation}

Following \cite{ReichelMendez}, the property that the constant charge density is an equilibrium distribution can be rewritten as the equation 
\begin{equation}
\label{reichel-mendes-eq}
\mathcal{N}_\phi +\frac{1}{2} = 0 \qquad \text{on $Z_\phi$.}  
\end{equation}

Indeed, in this case it follows from (\ref{eq:jump-condition}) that $\frac{\partial u_\phi}{\partial \nu^-}\equiv 0$ on $Z_\phi$ and therefore $u_\phi$ equals a constant on $\overline D_\phi \supset Z_\phi$ since it is harmonic in $D_\phi$. More precisely, since $0\in D_{\vp}$ and  $u_\phi(0)=0$ by definition of $u_\phi$, we deduce that that $u \equiv 0$ in $D_{\vp}$. It follows that the function $u_\phi$ satisfies the overdetermined problem (\ref{eq: Poisson-Dir-exceptional}) on $\Omega= \Omega_\phi$ in this case. Hence $\Omega_\phi$ is an exceptional domain provided that $u_\phi$ does not change sign.  To see this latter fact, we estimate 
\begin{align}
u_\vp(z,t)&= \int_{Z_\phi} \left( \frac{1}{(|z-y'|^2 +(s-t)^2  )^{1/2}}-\frac{1}{|y|} \right) d\sigma(y) \nonumber\\
&\leq \int_{Z_\phi} \left( \frac{1}{((|z| /2)^2 +(s-t)^2  )^{1/2}}- \frac{1}{|y|} \right) d\sigma(y)\qquad \text{for  $|z|\geq 3 \|\vp\|_{L^\infty(\R)}$,} \label{eq:est-u-vp-intro}
\end{align}
where we have written $y = (y',s)$ with $y' \in \R^2, s \in \R$. Since the integrand in (\ref{eq:est-u-vp-intro}) is monotone decreasing in $|z| \ge 3 \|\vp\|_{L^\infty(\R)}$ and tends to the function $y \mapsto -\frac{1}{|y|}$ as $|z| \to \infty$, monotone convergence yields
\begin{equation}
  \label{eq:u-phi-asymptotics}
\lim_{|z|\to \infty} u_\vp(z,t)=-\int_{Z_\phi}\frac{1}{|y|}d\sigma(y)= -\infty.
\end{equation}
Moreover, since $u_\vp$ is harmonic, it cannot attain a maximum in $\Omega_\phi$ unless it is constant, which is excluded by (\ref{eq:u-phi-asymptotics}). Since, in addition, $u_\vp=0$ on $Z_\vp$ and $u_\vp$ is $2\pi$-periodic in the $t$ direction, it follows that $u_\vp<0$ in $\O_\vp$.\\ 
Consequently, our problem is reduced to the problem of finding nonconstant $2 \pi$-periodic $C^2$ functions $\phi: \R \to (0,\infty)$ such that (\ref{reichel-mendes-eq}) holds on $Z_\phi$.  As noted already, the constant functions $\phi \equiv \lambda$, $\lambda>0$ are trivial solutions of (\ref{reichel-mendes-eq}).\\

The following is our main result.
 
\begin{Theorem}\label{theo1}
Let  $\alpha \in (0,1)$. There exists a strictly decreasing  sequence  $(\lambda_k)_{k\geq 1}$ with  $\lim \limits_{k \to \infty}\lambda_k= 0$  and the following properties: For each  $k \geq 1$, there exists $\e_k>0$  and a smooth map
\begin{align*}
 (-\e_k,\e_k) &\longrightarrow  (0,\infty) \times C^{2,\alpha}(\mathbb{R})\\
        s&\longmapsto (\lambda_k(s),\psi^k_s)
 \end{align*}
with  $\lambda_k(0)=\lambda_k$, $\psi^k_0\equiv 0$ and such that
for all
$s\in(-\e_k,\e_k),$ letting $\phi^k_s:=\lambda_k(s)+\psi^k_s$, the set
\begin{equation}\label{petcy}
\O_{\phi^k_s}= \biggl\{(z,t)\in \mathbb{R}^{2}\times\mathbb{R}:\quad
|z|>  \phi^k_s( t)\biggl\}
\end{equation}
is an  exceptional domain. In addition,  its boundary $Z_{\phi^k_s}=\de\O_{\phi^k_s}$ satisfies \eqref{reichel-mendes-eq}.

Moreover, for every $s \in (-\e_k,\e_k)$, the function $\psi^k_s$ is even and ${2\pi} $-periodic in
$t$. Furthermore, we
have
$$
\psi^k_s(t)=s\Bigl(\cos( k t)+\mu^k_s(t)\Bigr) \qquad \text{for $s \in (-\e_k,\e_k)$}
$$
with a smooth map $ (-\e_k,\e_k) \to  C^{2,\alpha}(\mathbb{R})$, $s\mapsto \mu^k_s$ satisfying
$$
\int_{0}^{2\pi} \mu^k_s(t) \cos( kt) \,dt = 0 \qquad \text{for $s \in (-\e_k,\e_k)$}
$$
and $\mu^k_0\equiv 0$.\\
\end{Theorem}

As mentioned earlier, since the exceptional domains constructed in Theorem  \ref{theo1} satisfy (\ref{eq:op-eq}), they allow the interpretation that the constant charge distribution on $\partial \O_{\phi_s^k}$ is an electrostatic equilibrium. Physically, this means there is no potential difference  between any two points on  the boundary $\partial \O_{\phi^k_s}$ and therefore no electric current which could alter the constant charge distribution \cite{Wermer, ReichelMendez}.\\   

To solve  (\ref{reichel-mendes-eq}), we first need to write this  equation as a functional equation in suitable function space. For $j = 0,1,2$, we let $C_{p,e}^{j,\a}(\R) $ denote the space of even and $2\pi$ periodic functions in $C^{j,\a}(\R)$. Moreover, we let $\cU$ denote the cone of strictly positive functions in
$C^{2,\alpha}_{p,e}(\R)$. We can then write (\ref{reichel-mendes-eq}) as a functional equation in the unknown $\phi \in \cU$ of the form 
\begin{equation}
  \label{eq:op-eq}
H(\phi) +2 \pi = 0 \qquad \text{in $C_{p,e}^{2,\a}(\R)$,} 
\end{equation}
where 
$H: \cU \to C_{p, e}^{1,\a}(\R),$  is a  nonlinear operator defined by
\begin{equation}
  \label{eq:def-H-intro}
H(\phi)(t)= 4 \pi\, \mathcal{N}_\phi(\phi(t)\sigma, t)
\qquad \text{for $\phi \in 
\cU,$ $t \in \R$.}
\end{equation}
Here $\sigma \in S^1$ can be chosen arbitrarily by the rotational invariance of $\mathcal{N}_\phi$. We can further rewrite (\ref{eq:op-eq}) as a bifurcation equation of the form 
\begin{equation}
  \label{eq:bif-eq}
\Phi(\lambda,\phi)=0 \qquad \text{in $C_{p,e}^{1,\a}(\R)$,}
\end{equation}
where $\Phi$ is defined on an open subset of $\R \times C_{p,e}^{2,\alpha}(\R)$ by
\begin{equation} \label{eq:defGGf-eq}
\Phi(\lambda, \phi):= H(\lambda + \phi) +2 \pi.
\end{equation}
With the help of bifurcation theory, we establish the existence of nonconstant solutions $\phi$ close to the trivial branch of solutions $\{(\lambda,0)\::\: \lambda >0\} \subset \R \times C_{p,e}^{2,\alpha}(\R)$. \\

We emphasize that the function $\mathcal{N}_\phi$ has similarities with the so-called {\em nonlocal mean curvature functions of order $\alpha \in (0,1)$ on $Z_\phi$}, which are given by 
\begin{equation}
  \label{eq:nonlocal-mean-curvature}
  x \mapsto c_{\alpha} \int_{Z_\phi} \frac{(y-x) \cdot \nu^+(y)}{|x-y|^{3+\alpha}} d \sigma(y),
\end{equation}
where $c_{\alpha}$ is a constant. The problem of constructing surfaces of the type $Z_\phi$ with constant nonlocal mean curvature has been considered recently by Cabr\'e, Fall and Weth in \cite{CMT}, and Minlend, Niang  and Thiam \cite{Mi.Al.Th.}. 
The difference between (\ref{eq:electric-field})  and (\ref{eq:nonlocal-mean-curvature}) is two-fold. First, the case $\alpha=0$ corresponds to a limiting case which is not admitted in \cite{CMT, Mi.Al.Th.} but more closely related to the $0$-fractional perimeter, which has been studied recently in \cite{DNP}. Moreover, the normal $\nu^+$ is taken at $x$ in (\ref{eq:electric-field}), while it evaluated at $y$ in (\ref{eq:nonlocal-mean-curvature}).

Despite these differences which strongly affect the analysis, the construction  here follows the general strategy of \cite{CMT, Mi.Al.Th.} which is based on an application of the Crandall-Rabinowitz bifurcation theorem, see \cite[Theorem 1.7]{M.CR}. The main difficulties in this approach are the following. On a technical level, proving sufficient regularity of the operator $H$ in (\ref{eq:def-H-intro}) as a map between Hölder spaces is not straightforward since we are dealing with a quasilinear hypersingular integral operator. To prove smoothness of $H$, we follow the approach in\cite{CMT}. We note that we need $H$ to be at least of class $C^2$ to apply the Crandall-Rabinowitz theorem. The second main difficulty is to verify the functional analytic properties of the linearization of $H$ at constant functions $\lambda>0$ which guarantees that the problem admits transversal bifurcation from a simple eigenvalue in the sense of Crandall and Rabinowitz \cite{M.CR}. With regard to this aspect, the key steps of the proof are the representation formula in Lemma~\ref{lem:eigenval}, the transversality property given in Lemma~\ref{derv222} and the Schauder type regularity property in Lemma~\ref{LemmTworegul}. 

At this point, we wish to remark that the Crandall-Rabinowitz bifurcation theorem and other related 
tools in topological bifurcation theory are commonly used in the study of overdetermined boundary value problems. In particular, we wish to mention the references \cite{Fall-MinlendI-Weth, RosRuizSicbaldi, Morabito Sicbaldi, LiuWangWei, F.Morabito1, F.Morabito2}, where similar results are obtained for overdetermined boundary value problems related to different types of linear and nonlinear elliptic equations.

The paper is organized as follows:   In Section \ref{s:N-ope},  we derive suitable representations of the operator $H$ which allow to study its regularity and its linearization at constant functions. In Section \ref{sec: linearized}, we then study the spectral properties of the  linearized operator of $H$ at a  constant function $\l>0$ and derive  qualitative properties of its eigenvalues, as $\l$ varies.  In Section \ref{ss:Existt-d-per-excepti}, we complete the proof of  Theorem \ref{theo1} based on the Crandall-Rabinowitz theorem. In Section~\ref{sec: regular}, we prove smoothness of the operator $H$ in (\ref{eq:def-H-intro}). Finally, in the appendix of this paper, we collect some useful properties of modified Bessel functions.

\bigskip
\noindent \textbf{Acknowledgements}: 
M.M. Fall  and I.A. Minlend are   supported by the Alexander von Humboldt foundation.
Part of this work was  carried  out when  I.A. Minlend and M.M. Fall  were visiting the Goethe University Frankfurt am Main. They are gratefully to the  Mathematics department    for the hospitality.


\section{Representations of the nonlocal nonlinear operator $H$}
\label{s:N-ope}
We fix $\alpha \in (0,1)$ in the following.   For  $j=0,1,2$, we consider the Banach space
\begin{align*}
&C_{p,e}^{j,\alpha}(\R):=\bigl\{  u \in C^{j,\alpha}(\R) : \textrm{$u$ is even and $2\pi$-periodic}\bigl\},
\end{align*}
 and let 
$$
\cU:= \{ \phi \in C^{2,\alpha}_{p,e}(\R) \::\: \min_{\R} \phi>0 \}.
$$
For a function $\phi \in \cU$, we define the domain 
$$
\O_\phi:=\left\{ (z,t)\in
\mathbb{R}^{2}\times\mathbb{R}\,:\, |z|>\phi(t)  \right\}. 
$$
The  boundary 
$$
Z_{\phi}=\left\{(\phi(s) \s, s)\in  \R^2 \times \R  \, : \, \s\in  S^1\right\}
$$  is parametrized by

\begin{equation}\label{eq:parame}
F_{\phi}:  \R\times S^1 \to  \de \O_{\phi}, \qquad F_{\phi}(s, \sigma)= (\phi (s) \sigma, s ),
\end{equation} 
and the  unit normal on $Z_\phi$ pointing inside $\Omega_\phi$ is  given by
\begin{equation}\label{eq:unitinnorm}
\nu_{Z_{\phi}}^+(F_{\phi}(s,\th))=\frac{1}{\sqrt{1+(\phi')^2(s)}}(\th, {-\phi'}(s)) \qquad \textrm{ for } s\in\R,\, \th\in S^1. 
\end{equation}
We recall the definition of the nonlinear operator $H$ given by \eqref{eq:electric-field} and (\ref{eq:def-H-intro}), which can be written as 
\begin{equation}\label{eq:electric-op}
H(\phi)(s):= 4 \pi \mathcal{N}_\phi(s,F_\phi(s,e_1)) = \int_{Z_\phi} \frac{(y-F_\phi(s,e_1)) \cdot \nu^+(F_\phi(s,e_1))}{|F_\phi(s,e_1)-y|^3} d \sigma(y) 
\end{equation}
for $\phi \in \cU$, $s \in \R$. The following lemma provides a more explicit representation of $H$ as an integral operator. 
\begin{Lemma}\label{lem:expreH1}
Let $\phi \in \cU$. Then for $s \in \R$ we have $H(\phi)(s)  = -\frac{1}{\sqrt{1+(\phi')^2(s)}} \cH(\phi)(s)$ with 
\begin{align}\label{ewp1H}
\cH(\phi)(s)&=  \int_{\R} \int_{S^1} \frac{\bigl(\phi(s)-\phi(s-t)- t \phi'(s)\bigl)\phi(s-t)\sqrt{1+(\phi')^2(s-t)} }{\{t^2+ (\phi(s)-\phi(s-t))^2 + \phi(s)\phi(s-t)|\s- e_1|^2 \}^{\frac{3}{2}}}d\s dt \nonumber \\\
  &+ \frac{1}{2} \int_{\R} \int_{S^1} \frac{\phi^2(s-t)|\s- e_1|^2\sqrt{1+(\phi')^2(s-t)} }{\{t^2+ (\phi(s)-\phi(s-t))^2+ \phi(s)\phi(s-t)|\s- e_1|^2 \}^{\frac{3}{2}}}d\s dt.
\end{align}
\end{Lemma} 

\proof
Using  \eqref{eq:electric-op} and the  parameterization  \eqref{eq:parame}, we have
\begin{align}
  H(\phi)(s)
&=   \int_{\R}\int_{S^1} \frac{\left\{F_{\phi}(\bar{s}, \s)-F_{\phi}(s,e_1)    \right\} \cdot\nu_{S_{\phi}}^+(F_{\phi}(s,e_1))  }{ |F_{\phi}(s,e_1)-F_{\phi}(\bar{s},\s)|^{3}}  J_{F_{\phi}}(\bar{s}, \s)\,d \bar s d\s,  \label{eq-1-H(u)}
 \end{align}
where
 $$
J_{F_{\phi}}(\bar{s}, \s)=   \phi(\bar{s})\sqrt{1+(\phi')^2(\bar{s})}  \quad \textrm{ for } \bar{s}\in\R,  \s\in S^1.
$$
We also note that for $s, \bar s \in \R$ and $\s \in  S^1$ we have
\begin{align*}
|F_{\phi}(s,e_1)-F_{\phi}(\bar{s},\s)|^2 &= |s-\bar s|^2 + | \phi(s)e_1 - \phi(\bar s) \s|^2\\
& = |s-\bar{s}| ^2+ (\phi(s)-\phi(\bar s))^2 +2 \phi(s)\phi(\bar s)(1-   e_1 \cdot  \sigma )
\end{align*}
and
\begin{align*}
\left\{F_{\phi}(\bar{s}, \s)-F_{\phi}(s,e_1)    \right\} \cdot\nu_{S_{\phi}}^+(F_{\phi}(s,e_1)) &= \frac{(s-\bar s) \phi'(s)-(\phi(s)e_1 - \phi(\bar s) \s)e_1  }{\sqrt{1+(\phi')^2(s)}}   \\
&=  \frac{\phi(\bar{s})-\phi(s)+ (s-\bar s)\phi'(s)-\phi(\bar{s})(1-e_1\cdot  \sigma)}{\sqrt{1+(\phi')^2(s)}}.
\end{align*}
Inserting these identities in \eqref{eq-1-H(u)}, we find that $\mathcal{H}(\phi)(s)=-\sqrt{1+(\phi')^2(s)}H(\phi)(s)$ satisfies
 \begin{align*}
   &\mathcal{H}(\phi)(s)=- \int_{\R} \int_{S^1} \frac{\bigl(\phi(\bar{s})-\phi(s)+ (s-\bar s) \phi'(s)-\phi(\bar{s})(1-e_1\cdot \sigma)\bigl)\phi(\bar{s})\sqrt{1+(\phi')^2(\bar{s})} }{\{|s-\bar{s}| ^2+ (\phi(s)-\phi(\bar s))^2 +2 \phi(s)\phi(\bar s)(1-   e_1\cdot  \sigma ) \}^{\frac{3}{2}}}d\s d\bar{s}\\
   &=  \int_{\R} \int_{S^1} \frac{\bigl(\phi(s)-\phi(s-t)- t \phi'(s)+\phi(s-t)(1-e_1\cdot \sigma)\bigl)\phi(s-t)\sqrt{1+(\phi')^2(s-t)} }{\{t^2+ (\phi(s)-\phi(s-t))^2 +2 \phi(s)\phi(s-t)(1-  e_1\cdot  \sigma ) \}^{\frac{3}{2}}}d\s dt\\
   &= \int_{\R} \int_{S^1} \frac{\bigl(\phi(s)-\phi(s-t)- t \phi'(s)+\phi(s-t)(1- \sigma_1)\bigl)\phi(s-t)\sqrt{1+(\phi')^2(s-t)} }{\{t^2+ (\phi(s)-\phi(s-t))^2 +2 \phi(s)\phi(s-t)(1-  \sigma_1 ) \}^{\frac{3}{2}}}d\s dt.
 \end{align*}
Here,  the second equality follows from the change of variable $t= s-\bar s$. Moreover, we use $1-\s_1=\frac{|\s-e_1|^2}{2}$ to get (\ref{ewp1H}) from the third equality. 

\QED

In the next lemma, we rewrite the representation of the operator $H$ from Lemma~\ref{lem:expreH1} in a somewhat more convenient form, which will then allow us to study its regularity and to compute its linearization at constant functions.

We need some preliminary definitions and observations. First we set
$$
p_{\s}:=|\s-e_1| \qquad \text{for $\s \in S^1$}
$$
and  define the  maps $\L_0,  \L_1: C^{2,\alpha}(\R)\times\R^2 \to \R$ by
\begin{align}
\L_0(\phi,s,t)&:= \frac{\phi(s)-\phi(s-t)}{t} = \int_0^1 \phi' (s-\rho t) d\rho ,  \label{eq:def-Lamb-0}\\
\L_1(\phi, s,t)&:= \L_0(\phi,s,t)- \phi' (s) =  \int_0^1( \phi' (s-\rho  t)- \phi' (s) )d\rho.\label{eq:def-Lamb-2}
\end{align}
With these  notations, we get the follow alternative representation of the operator $H$ in Lemma  \ref{ewp1H}.

\begin{Lemma}\label{lem:expreH2}
For every $\phi \in \cU$ and $s\in \R$, we have $H(\phi)(s)  = -\frac{1}{\sqrt{1+(\phi')^2(s)}} \cH(\phi)(s)$ with 
\begin{align}\label{expreH2}
\cH(\phi)(s)&= \int_{S^1} \frac{1}{p_\s}  \int_{\R}\mathcal{M}(\phi, s,t, p_\s)dt d \s+ \frac{1}{2} \int_{S^1}  \int_{\R} \mathcal{\ov{M}}(\phi, s,t, p_\s)dt d \s 
\end{align}
with
\begin{equation}
  \label{eq:def-M-bar-M}
\begin{aligned}
\mathcal{M}(\phi, s,t, p) &:=t \L_1(\phi, s,t p) \mathcal{{K}}(\phi, s, t,tp),\\
\mathcal{\ov{M}}(\phi, s,t, p) &:=\phi(s-t p) \mathcal{{K}}(\phi, s,t, tp)
\end{aligned}
\end{equation}
and
\begin{equation}
  \label{eq:def-K}
\mathcal{K}(\phi, s,t, \xi):=\frac{\phi(s-\xi) \sqrt{1+(\phi')^2(s-\xi)}}{\{t^2+ t^2 \L^2_0(\phi,s,\xi)+ \phi(s)\phi(s-\xi)\}^{\frac{3}{2}}}.
\end{equation}
\end{Lemma}

\proof
The expression \eqref{expreH2} follows from \eqref{eq:def-Lamb-2} and \eqref{ewp1H} with the change of variables $\t=\frac{t}{p_\s}.$
\QED

\begin{Remark}
\label{absolute-convergence-constant-functions}{\rm 
(i) It does not appear obvious that the integrals in Lemma~\ref{lem:expreH1} and Lemma~\ref{lem:expreH2} converge absolutely in Lebesgue sense. To see this, it suffices to consider the first integral in \eqref{expreH2}, since second integral in \eqref{expreH2} is clearly convergent. We let  $\delta := \min_{\R} \phi>0$ and we note that
\begin{equation}
  \label{eq:theta-formula}
p_\s=|\s-e_1|=2\sin(\frac{|\theta|}{2}) \qquad \text{for $\theta \in [-\pi,\pi]$ and $\sigma =(\cos(\theta), \sin(\theta))\in S^1$.} 
\end{equation}
As a consequence,
\begin{equation}
  \label{eq:theta-sigma-formula}
  \int_{S^1}g(p_\s)d\sigma = 2 \int_0^\pi g\bigl(2\sin(\frac{\theta}{2})\bigr)d\theta = 4 \int_0^{\frac{\pi}{2}} g(2\sin \theta) d\theta
\end{equation}
for every measurable function $g$ on $(0,2)$ for which the integral on the RHS exists in Lebesgue sense.  Then \eqref{eq:def-Lamb-2} implies that
\begin{align}\label{esgamma1}
| \L_1(\phi, s,t p)| \le  2 \|\phi \|_{C^{1,\alpha}(\R)}\,\min(1,|p|^{\alpha}|t|^{\alpha})
\end{align}
and  we have
\begin{align}\label{CoverLeb1}
&\int_{S^1} \frac{1}{p_\s}  \int_{\R} | \t|\left| \L_1(\phi, s,\t p_\s)\right| \phi(s-\t p_\s)\mathcal{K}(\phi, s,\t, \t p_\s) d\t d\s\nonumber\\
& \leq C  \int_{S^1} \frac{1}{p_\s^{1-\alpha}} d\s   \int_{\R}  \frac{|\tau|^{1+\alpha}}{\{\tau^2+ \delta^2\}^{\frac{3}{2}}}d\t,
\end{align}
where $C$ is a constant only depending on $\|\phi \|_{C^{1,\alpha}(\R)}$.  Applying   \eqref{eq:theta-sigma-formula}  with $g(x)=x^{\alpha-1}$, we find 
$$
\int_{S^1} \frac{1}{p_\s^{1-\alpha}} d\s = 2^{1+\alpha} \int_{0}^{\frac{\pi}{2}} \sin^{\alpha-1}(\theta) d\theta<\infty,
$$
since $\alpha<1$.  The  second integral in \eqref{CoverLeb1} is  also  finite and we deduce the convergence of the first integral in \eqref{expreH2}.

(ii) From Lemma \ref{lem:expreH2}, we have, for all $\lambda>0$,
\begin{equation}
  \label{eq:2pi--const-eq}
H(\lambda)(s)= -\frac{1}{2}\int_{\R} \int_{S^1} \frac{\lambda^2 }{\bigl(t^2+\lambda^2 )^{\frac{3}{2}}}  d\s dt=-2\pi\int_{0}^{+\infty}  \frac{1}{\bigl(1+t^2)^{\frac{3}{2}}} dt = -2\pi
\end{equation}
To see the last equality, we note that, by the change of variable $\tau = \frac{1}{1+t^2}$, we have, more generally, 
\begin{equation}
  \label{eq:beta-function-eq}
  \int_{0}^{+\infty}  \frac{1}{\bigl(1+t^2\bigr)^{s}} dt  = \frac{1}{2}
  \int_0^1 \tau^{s-\frac{3}{2}}(1-\tau)^{-\frac{1}{2}}d \tau =  \frac{1}{2}B(s-\frac{1}{2},\frac{1}{2}) \quad \text{for $s > \frac{1}{2}$}
\end{equation}
with the Beta-Function $B$. We shall use this identity late for different values of $s$ later. It follows from (\ref{eq:2pi--const-eq}) that $H(\lambda) \equiv -2\pi$ on $\R$ and, consequently,
\begin{equation}\label{HatCos}
\Phi(\lambda, 0) = 0 \quad \textrm{ for all }\quad \lambda> 0
\end{equation}
with $\Phi$ defined in (\ref{eq:defGGf-eq}). Hence, as pointed out already in the introduction, (\ref{eq:bif-eq}) has the trivial solution branch $\{(\lambda,0)\::\: \lambda>0\}$.
}
\end{Remark}

\section{Regularity of the operator  $H$}
\label{sec: regular}
In this section, we prove the regularity of the operator $H$ using the  expression in Lemma \eqref{lem:expreH2}. The proof follows the strategy of the argument  in \cite[Section 4]{CMT} for a nonlocal mean curvature operator, but some modifications are necessary due to the different form of the underlying operators and due to the fact that we need to show that $H$ is smooth as an operator from $C^{2, \alpha}(\R)$ to  $C^{1, \alpha}(\R)$.

\subsection{Preliminaries}\label{ss:diff-cal}

For a finite set $\cA$, we let $|\cA|$ denote the length (cardinal) of $\cA$. It will be understood that $ |\emptyset|=0$. 
Let $Z$ be a  Banach space and $U$ a nonempty open subset of $Z$. If $T \in C^{k}(U,\R)$ and $\phi \in U$, then $D^kT(\phi)$ is a continuous 
symmetric $k$-linear form on $Z$ whose norm is given by  
$$
   \|D^{k}T ({\phi}) \|= \sup_{{v}_{1}, \dots,  {v}_{k}\in Z }
     \frac{|D^{k} T ({\phi})[v_1,\dots,v_k]| }{   \prod_{j=1}^k \|   {v}_{j} \|_{ Z }}    .
$$
If  $T_1,\, T_2 \in C^k(U,\R)$,  then also $T_1 T_2 \in C^k(U,\R)$, and the $k$-th derivative of $T_1 T_2$ at $u$ is given by 
\be \label{eq:Dk-T1T2}
D^k(T_1 T_2 )({\phi})[v_1,\dots,v_k]= \sum_{\cA \in  \scrS_k} D^{|\cA|} T_1({\phi})[v_i]_{i\in \cA} \,  D^{k-|\cA|} T_2({\phi}) [v_i]_{i\in \cA^c} ,
\ee
where  $\scrS_k $ is the set of subsets of $\{1,\dots, k\} $ and  $ \cA^c= \{1,\dots, k\}\setminus \cA$ for $\cA\in \scrS_k$.  If, in particular, $L: Z\to \R$ is a linear map, we have 
\be \label{eq:Dk-LT2}
D^{k}(L T_2 )({\phi})[v_i]_{i \in \cA_k}= L({\phi})    D^{k} T_2({\phi})[v_i]_{i \in \cA_k} +  
\sum_{j=1}^k L({v}_j)   D^{k-1} T_2({\phi})  [v_i]_{\stackrel{i \in \cA_k}{i \not = j}},
\ee
where $\cA_k:= \{1,\dots,k\}$. 
 
\subsection{Regularity of the operator  $H$}\label{ss:regul-H}
To prove the regularity of the operator $H$ in (\ref{eq:def-H-intro}) on the open set $\cU$, it suffices to fix $\delta>0$
and to show the regularity of the operator 
\begin{align}\label{lem:DERVA1}
\mathcal{H}:  \cO_\d \to C^{1,\a}(\R),\qquad   \mathcal{H}(\phi)(s)&:=-\sqrt{1+(\phi')^2(s)}H(\phi)(s).
\end{align}
where $\cO_\d:= \{\phi \in C^{2,\a}_{p,e}(\R)\::\: \min_{\R} \phi > \delta\}$. The following is the main result of this section.

\begin{Proposition}
\label{prop:smooth-ovH}
For every  $\d>0$, the map  $\cH: \cO_{\d} \subset C^{2,\a}(\R) \to C^{1, \a}(\R)$ defined by \eqref{lem:DERVA1} is of class $C^\infty$, and for every  $k\in\N$, and $v_1,\dots,v_k \in C^{2,\a}(\R)$, we have
 \begin{align}
D^k   \mathcal{H} ({\phi})[v_1,\dots,v_k]  (s)&=\int_{S^1} \frac{1}{p_\s}  \int_{\R}  D^k_\phi \mathcal{M}({\phi} , s,t,p_\s)[v_1,\dots,v_k]\,dt d\s  \label{eq:DkH-in-integ} \\
 &+\frac{1}{2} \int_{S^1} \int_{\R}D^k_\phi \mathcal{\ov{M}}({\phi},s,t,p_\s)[v_1,\dots,v_k]\,dt d\s,
\end{align}
where $\mathcal{M}$ and $\mathcal{\ov{M}}$ are defined in (\ref{eq:def-M-bar-M}).
\end{Proposition}

The remainder of this section is devoted to the proof of Proposition~\ref{prop:smooth-ovH}. For this, we need to decompose the terms appearing in the integral formula (\ref{expreH2}) for the operator $\cH$ in a suitable way into even and odd parts which we then estimate separately. For this, we first note that
$$
t \L_1(\phi,t,t p) = \frac{1}{2p}\bigl(L_e(\phi,t,tp)+L_o(\phi,t,tp)\bigr)
$$
for $\phi \in \cO_\d$, $s,t \in \R$, $p \ge 0$ with
\be
\label{eq:defLe_Lo}
\begin{aligned}
L_e(\phi,s,t )&= t \Bigl(\L_1(\phi,s,t)+\L_1(\phi,s,-t)+2 \phi'(s)\Bigr)= 2\phi(s)-\phi(s-t)-\phi(s+t),\\  
L_0(\phi,s,t)&= t \Bigl(\L_1(\phi,s,t)-\L_1(\phi,s,-t)-2 \phi'(s)\Bigr)= \phi(s+t)-\phi(s-t)-2 t\phi'(s) 
\end{aligned}
\ee
and thus, by Lemma \ref{lem:expreH2},
\begin{align}\label{lem:DERV2}
 &\mathcal{H}(\phi)(s)= \frac{1}{2} \int_{S^1} \frac{1}{p_\s^2}  \int_{\R} \Bigl( L_e(\phi,s,t p_\s)+L_o(\phi,s,t p_\s)\Bigr)  \mathcal{K}(\phi, s,t,tp_\s) dt d \s \nonumber\\
                      &+ \frac{1}{2} \int_{S^1}  \int_{\R} \phi(s-t p_\s) \mathcal{K}(\phi, s,t,tp_\s)dt d \s, \nonumber \\
                     &= \frac{1}{2} \int_{S^1} \frac{1}{p_\s^2}  \int_{\R} L_e(\phi,s,t p_\s)\mathcal{K}(\phi, s,t,t p_\s) dt d \s  \\
& +\frac{1}{2}  \int_{S^1} \frac{1}{p_\s^2}  \int_{\R} L_o(\phi,s,tp_\s)\mathcal{K}_o(\phi, s,t,t p_\s)dt d \s \nonumber 
  + \frac{1}{2} \int_{S^1}  \int_{\R} \phi(s-t p_\s) \mathcal{K}(\phi, s,t,t p_\s)dt d \s, 
\end{align}
where $\mathcal{K}$ is defined in (\ref{eq:def-K}) and $\mathcal{K}_o$ is, the odd part of $\mathcal{K}$ in its last variable,  given by
\begin{align}
  \mathcal{K}_o(\phi, s,t, \xi)&= \frac{1}{2}\Bigl(\mathcal{K}(\phi, s, t, \xi)-\mathcal{K}(\phi, s,t,-\xi)\Bigr) . \label{eq:def-K-o}
\end{align}

We have the following result.
\begin{Lemma}\label{lem:I_eI_o}
Let $m\in \{0,1\}$,  $\mu\ \in C^{m,\a}(\R^2)$ and  $\nu\ \in C^{m,\a}(\R^3)$  such that  $\nu(\cdot,\cdot, 0 )=0$ and  for all $i\leq m$, 
\be\label{eq:desimunu}
\|\de_s^i\mu (\cdot,t)\|_{C^{\a}(\R)}  + \|\de^i_{s} \nu (\cdot,t,\cdot )\|_{C^{\a}(\R^2)}\leq \frac{C_0}{(1 + t^2)^{3/2}}.
\ee
 Define, for $\phi\in C^{m+1,\a}(\R)$, 
$$
I_e(s):= \int_{S^1}\frac{1}{p_\s^2}\int_{\R} L_e(\phi,s,t p_\s)\mu(s,t) dt d\s 
$$
and 
$$
 I_o(s):=\int_{S^1}\frac{1}{p_\s^2}\int_{\R} L_o(\phi,s,t p_\s)\nu(s, t,tp_\s) dt d\s. 
$$
Then
$$
\|I_e\|_{C^{m,\a}(\R)}+ \|I_o\|_{C^{m,\a}(\R)}  \leq C C_0 \|\phi \|_{C^{m+1,\a}(\R)},
$$
for some $C=C(m,\a)$.
\end{Lemma}
\begin{proof}
We consider the case $m=0$ as the proof will be similar in the case $m=1$. By linearity, we assume that  $ \|\phi \|_{C^{1,\alpha}(\R)}\leq 1$. We start with $I_e$. We have  
\be \label{eq:Le1}
|L_e(\phi,s,tp)|\leq  4    |t|^{1+\a} |p|^{1+\a} ,
\ee
\be \label{eq:Le2}
|L_e(\phi,s,tp)- L_e(\phi,\ov s,tp)|\leq  4  |p|\min (|s-\ov s|^\a, |p|^\a) \min(|t|, |t|^{1+\a}) 
\ee
and 
\be \label{eq:Le4}
|L_e(\phi,s,tp)- L_e(\phi,\ov s,tp)|\leq  4   |s-\ov s| |t|^{\a} |p|^{\a}.
\ee
Now using  \eqref{eq:Le1} and \eqref{eq:desimunu}, we get 
\be\label{eq:Le3}
\int_{\R}  |L_e(\phi,s,tp)| \mu(s,t)\, dt \leq  C C_0 |p|^{1+\a}  \int_{\R}  \frac{|t|^{1+\a}dt}{(1 + t^2)^{3/2}}\leq C C_0   |p|^{1+\a} .
\ee
From this,    we deduce that 
\begin{align}\label{eq:Ies1}
|I_e(s)|\leq  C C_0\int_{S^1} p_\s^{-1+\a} d\s\leq C C_0  .
\end{align}
Similarly, using  \eqref{eq:uv-s_1s_2}, \eqref{eq:Le2}, \eqref{eq:Le3}, \eqref{eq:Le4}  and  \eqref{eq:desimunu}, we have 
\begin{align}\label{eq:keyesti}
&|I_e(s) - I_e(\ov s)  |\leq C C_0\int_{\{\s\in S^1,\,\, p_\s \leq |s-\ov s|\}}p_\s^{-1}    \min (|s-\ov s|^\a,  p_\s^\a)  d\s   \int_{\R}  \frac{|t|^{1+\a}dt}{(1 + t^2)^{3/2}}\nonumber\\
&+ C C_0 |s-\ov s|  \int_{\{\s\in S^1,\,\, p_\s \geq |s-\ov s|\}}   p_\s^{-2+\a} d\s     \int_{\R}  \frac{|t|^{1+\a}dt}{(1 + t^2)^{3/2}}\nonumber\\
&+ C C_0 |s-\ov s|^\a \int_{S^1} p_\s^{-1+\a} d\s  \int_{\R}  \frac{|t|^{1+\a}dt}{(1 + t^2)^{3/2}}\nonumber\\
&\leq  CC_0  \left(  |s-\ov s|^\a +  |s-\ov s| \right).
\end{align}
Combining  \eqref{eq:keyesti} with \eqref{eq:Ies1} we get 
\be\label{eq:i_e-ok}
\| I_e\|_{C^{\a}(\R)}\leq C C_0. 
\ee 
To estimate $I_o$, we write  $L_o(\phi, s,tp)=tp \int_{-1}^1\{\phi'(s+\t tp)-\phi'(s)\}\,d\t$ so that 
\begin{align}\label{esgamma22}
  |L_o(\phi, s,tp)| &\le  4  |t|^{1+\a} |p|^{1+\a} 
\end{align}
 and 
\begin{align}\label{eq:Phis1s2}
|  L_o(\phi, s,t p)- L_o(\phi, \ov s,tp)| & \leq  4   |t| |p|   |s-\ov s|^\a. 
\end{align}
On the other hand, since $\nu(s,t, 0)=0$, we get from  \eqref{eq:desimunu}, 
\be\label{eq:nu1}
|\nu(s,t, t p) |\leq C C_0   \frac{ p^\a  |t|^\a  }{(1 + t^2)^{3/2}}
\ee
and 
\be \label{eq:nu2}
| \nu(s, t,t p)-\nu(\ov s,t, t p) |\leq C C_0 \frac{  |s-\ov s|^\a  }{(1 + t^2)^{3/2}}   .
\ee 
Now using  \eqref{esgamma22} and \eqref{eq:desimunu}, we obtain
\begin{align}\label{I_o-L-inf}
|I_0(s)|\leq C C_0 \int_{S^1} p_\s^{-1+\a} \, d\s \int_{\R}  \frac{|t|^{1+\a}dt}{(1 + t^2)^{3/2}}\leq C C_0.
\end{align}
Moreover, using  \eqref{esgamma22}, \eqref{eq:Phis1s2},   \eqref{eq:nu1} and  \eqref{eq:nu2}, we get 
\begin{align*}
&|I_0(s)-I_0(\ov s)|\leq C C_0  |s-\ov s|^\a \int_{S^{1}} p_\s^{-1+\a} \int_{\R}  \frac{|t|^{1+\a}dt}{(1 + t^2)^{3/2}}
\leq C C_0  |s-\ov s|^\a .
\end{align*}
From this and \eqref{I_o-L-inf}, we obtain the estimate of $I_o$ and thus  by \eqref{eq:i_e-ok} the proof is complete. 
\QED
\end{proof}

The following Lemma provides  some estimates related to the kernel ${\cK}$ defined in (\ref{eq:def-K}). 
For a function $u: \R \to \R$, we  also set 
$$
[u; s_1,s_2]:= u(s_1)-u(s_2)\qquad \text{for $s_1,s_2 \in \R$,} 
$$
and we note that 
\be\label{eq:uv-s_1s_2} 
[u w; s_1,s_2] = [u;s_1,s_2]w(s_1) + u(s_2)[w;s_1,s_2] \qquad \text{for $u,w: \R \to \R$, $s_1,s_2 \in \R$.}
\ee 
With this, we have the following  estimate.
\begin{Lemma}\label{lem:est-cK} 
Let $k \in \N \cup \{0\}$,  $i\in \{0,1\}$ and $\cK$ defined in \eqref{eq:def-K}. Then, there exists a constant $c=c(\alpha,k,\d)>1$ such that  for all $(s,s_1,s_2,t,p,p')\in\R^6$ and ${\phi}\in \cO_\d$, we have 
   \be \label{eq:Dk-K-s}
\|\de_s^i  D_{\phi}^k {\cK}({\phi},s,t,tp   )   \|\leq   
\frac{c(1+ \|{\phi}\|_{C^{i+1,\a}(\R)} )^{c}   }{(1 + t^2)^{3/2}},
 \ee 
   \be  \label{eq:Dk-K-s_1s_2}
\| [ \de_s^i D_{\phi}^k {\cK}({\phi},\cdot,t ,tp);s_1,s_2] \|\leq   
\frac{c(1+ \|{\phi}\|_{C^{i+1,\a}(\R)} )^{c}   \, |s_1-s_2|^\a }{    (1 + t^2)^{3/2}}
 \ee
 and 
    \be  \label{eq:Dk-K-ov-p}
\| [ \de_s^i  D_{\phi}^k {\cK}({\phi},s,t , \cdot); \xi,\xi'] \|\leq   
\frac{c(1+ \|{\phi}\|_{C^{i+1,\a}(\R)} )^{c}   \, |\xi-\xi'|^\a   }{    (1 + t^2)^{3/2}}.     
 \ee
 \end{Lemma}
 \begin{proof}
Since the proof is similar to the one of \cite[Lemma 4.1]{CMT}, we will only sketch it. Throughout this proof, the letter $c$ stands for different constants greater than one and depending only on $\alpha, k$, and $\d$.
By definition, we have 
\begin{align}\label{eq:q2}
  { \mathcal{K}}({\phi} , s,t, \xi) =\phi(s-\xi)\sqrt{1+(\phi')^2(s-\xi)}\, \mathcal{\widetilde K}({\phi} , s, t, \xi)
\end{align}
with
$$
\mathcal{\widetilde K}({\phi} , s, t,\xi ) = \frac{1}{\{t^2+ t^2 \L^2_0(\phi,s, \xi)+ \phi(s)\phi(s-\xi)\}^{\frac{3}{2}}}. 
$$
By \eqref{eq:Dk-T1T2}, it is then enough to prove the estimates for $\mathcal{\widetilde K}$ in place of $\cK$.
Following the same argument as in the proof of   \cite[Lemma 4.1]{CMT}, we get \eqref{eq:Dk-K-s}, \eqref{eq:Dk-K-s_1s_2} and  \eqref{eq:Dk-K-ov-p} for $i=0$.
Now, for $i=1$, by direct computations, we have 
$$
 \de_s \mathcal{\widetilde K}(\phi,s, t,\xi)= \widehat{Q}(\phi,s,t, \xi)  \frac{1}{\{t^2+ t^2 \L^2_0(\phi,s,\xi)+ \phi(s)\phi(s-\xi)\}^{\frac{5}{2}}} 
$$
 with 
 $$
 \widehat{Q}(\phi,s,t, \xi):=-\frac{3}{2} t^2 D_\phi\L^2_0(\phi,s, \xi)[\phi']+ \phi'(s)\phi(s-\xi)+ \phi(s)\phi'(s-\xi)
 $$
By  \eqref{eq:Dk-T1T2} and the argument in the proof of  \cite[Lemma 4.1]{CMT}, we also get  the estimates of $\de_s D_\phi^k\mathcal{\widetilde K}$, so that  \eqref{eq:Dk-K-s}, \eqref{eq:Dk-K-s_1s_2} and  \eqref{eq:Dk-K-ov-p}  hold for $i=1$.
\QED
\end{proof}

We now  use (\ref{lem:DERV2}) to write 
\begin{align}
\mathcal{H}(\phi)(s)&= \frac{1}{2} \int_{S^1} \frac{1}{p_\s^2}  \int_{\R} \Bigl(\mathcal{M}_e(\phi, s,t, p)+\mathcal{M}_o(\phi, s,t, p)\Bigr)dt d \s \nonumber\\
&+ \frac{1}{2} \int_{S^1}  \int_{\R} \phi(s-t p_\s) \mathcal{K}(\phi, s,t,tp_\s)dt d \s, \label{lem:DERV2-1}
\end{align}
with $\mathcal{M}_e, \mathcal{M}_o: \cO_\d \times \R^3    \to \R$ defined by 
\be\label{eq:cMe} 
\mathcal{M}_e(\phi, s,t, p):=L_e(\phi, s,t p) {\mathcal{K}}(\phi, s,t, tp),  
\ee
\be \label{eq:cMo}  
\mathcal{M}_o(\phi, s,t, p):=L_o(\phi, s,t p) {\mathcal{K}}_o(\phi, s,t, tp),  
\ee
where  ${\mathcal{K}}$ is given in (\ref{eq:def-K}) and ${\mathcal{K}}_o$ is given in (\ref{eq:def-K-o}).
By \eqref{eq:Dk-LT2} we have
\begin{align}\label{eq:diffcM123-e}
&D^{k}_{\phi} \mathcal{M}_e(\phi, s,t, p)[v_i]_{i \in \cA_k}\\
  &=  L_e({\phi},s,tp)    D^{k}_{\phi} \cK({\phi},s,t,tp)[v_i]_{i \in \cA_k}+  \sum_{j=1}^k L_e({v}_{j},s,tp)   D^{k-1}_{\phi}  \cK({\phi},s,t,tp) 
[v_i]_{\stackrel{i \in \cA_k}{ i\neq j}}\nonumber
 \end{align}
and 
\begin{align}\label{eq:diffcM123-o}
&D^{k}_{\phi} \mathcal{M}_o(\phi, s,t, p)[v_i]_{i \in \cA_k}\\
  &=  L_o({\phi},s,tp)    D^{k}_{\phi} \cK_o({\phi},s,t,tp)[v_i]_{i \in \cA_k}+  \sum_{j=1}^k L_o({v}_{j},s,tp)   D^{k-1}_{\phi}  \cK_0({\phi},s,t,tp) 
[v_i]_{\stackrel{i \in \cA_k}{ i\neq j}},\nonumber
\end{align}
for $k \ge 1$ with $\cA_k= \{1,\dots,k\}$.

\begin{Lemma}\label{prop:smooth-ovH-0}
Let $k \in \N \cup \{0\}$, $\d>0$, $\phi \in \cO_\d$, and $v_1,\dots,v_k \in C^{2,\a}(\R)$. Moreover, let $\cF_e, \cF_o,\ti \cF: \R \to \R$ be defined by 
\begin{align*}
\cF_e(s)&= \int_{S^1} \frac{1}{p_\s^2}  \int_{\R}  D^k_\phi \mathcal{M}_e({\phi} , s,t,p_\s)[v_1,\dots,v_k]\,dt d\s,\\
\cF_o(s)&= \int_{S^1} \frac{1}{p_\s^2}  \int_{\R}  D^k_\phi \mathcal{M}_o({\phi} , s,t,p_\s)[v_1,\dots,v_k]\,dt d\s,\\
\ti \cF(s)&= \int_{S^1}   \int_{\R}D^k_\phi\{\phi(s-t p_\s) \mathcal{K}({\phi} , s,t,tp_\s)\}[v_1,\dots,v_k]\,dt d\s,
\end{align*}
where $\mathcal{M}_e, \mathcal{M}_o :  \cO_\d \times \R^3    \to \R$  are  defined by \eqref{eq:cMe}  and \eqref{eq:cMo}.
Then   $\cF_e, \cF_o, \ti \cF\in C^{1,\a}(\R)$. Moreover, there exists a constant $c=c(\a,k,\d)$ such that 
\begin{equation*}
\|\cF_e\|_{C^{1,\a}(\R)}+ \|\cF_o\|_{C^{1,\a}(\R)}+\|\ti \cF\|_{C^{1,\a}(\R)}   \leq  c(1+ \|{\phi}\|_{C^{2,\a}(\R)} )^{c}   \prod_{i=1}^k \|{v}_i\|_{C^{2, \a}(\R)} .
\end{equation*}
\end{Lemma}
\begin{proof}
 The estimates of  $\cF_e$ and  $\cF_o$ follow from Lemma \ref{lem:I_eI_o}, Lemma \ref{lem:est-cK}, \eqref{eq:diffcM123-e}.  \eqref{eq:def-K-o} and  \eqref{eq:diffcM123-o}. On the other estimate of  $\ti \cF$, follows from Lemma \ref{lem:est-cK} and the formula \eqref{eq:Dk-LT2}.
\QED
\end{proof}

Using the representation (\ref{lem:DERV2-1}) and Lemma \ref{prop:smooth-ovH-0}, we can now follow precisely the argument in the proof of \cite[Proposition 4.4]{CMT} to prove that the map  $\cH$ defined in (\ref{lem:DERVA1}) is of class $C^\infty$, and that for every  $k\in\N$, $v_1,\dots,v_k \in C^{2,\a}(\R)$, we have
 \begin{align}\label{eq:DkH-in-integ-variant}
   &D^k   \mathcal{H} ({\phi})[v_1,\dots,v_k]  (s)= \\
&\frac{1}{2}\int_{S^1} \frac{1}{p_\s^2}  \int_{\R}  \Bigl( D^k_\phi \mathcal{M}_e({\phi} , s,t,p_\s)[v_1,\dots,v_k]+ D^k_\phi \mathcal{M}_o({\phi} , s,t,p_\s)[v_1,\dots,v_k]\Bigr)dt d\s\nonumber\\
 &+\frac{1}{2}\int_{S^1}   \int_{\R}D^k_\phi\{\phi(s-t p_\s) \mathcal{K}({\phi} , s,t,tp_\s)\}[v_1,\dots,v_k]\,dt d\s.\nonumber
\end{align}
Here $ \mathcal{M}_e $, $ \mathcal{M}_o $ and $\mathcal{K}$ are defined in \eqref{eq:cMe}, \eqref{eq:cMo} and (\ref{eq:def-K}).

To finish the proof of Proposition~\ref{prop:smooth-ovH}, it only remains to observe that (\ref{eq:DkH-in-integ-variant}) is equivalent to (\ref{eq:DkH-in-integ}). This follows since
$$
\phi(s-t p) \mathcal{K}({\phi} , s,t,t p) = \mathcal{\ov{M}}({\phi} , s,t,p)
$$
and 
\begin{align*}
&\mathcal{M}_e({\phi} , s,t,p)+ \mathcal{M}_o({\phi} , s,t,p)\\  
  &= 2tp \L_1(\phi,t,t p)\mathcal{K}({\phi} , s,tp) - L_o(\phi,t,t p)\Bigl(\mathcal{K}({\phi} , s,t,tp)-\mathcal{K}_0({\phi} , s,t,tp)\Bigr)\\  
  &= 2p \mathcal{M}({\phi} , s,t, p) - \mathcal{\widetilde{M}}({\phi} , s,t,p), 
\end{align*}
where the quantity 
$$
\mathcal{\widetilde{M}}({\phi} , s,t,p) := L_o(\phi,t,t p)\Bigl(\mathcal{K}({\phi} , s,t,tp)-\mathcal{K}_0({\phi} , s,t, tp)\Bigr)
$$
is odd in $t$. Hence also $D^k_\phi \mathcal{\widetilde{M}}({\phi} , s,t,p)[v_1,\dots,v_k]$ is odd in $t$, and this implies that
\begin{align*}
  &\frac{1}{2}\int_{S^1} \frac{1}{p_\s^2} \int_{\R}  \Bigl( D^k_\phi \mathcal{M}_e({\phi} , s,t,p_\s)[v_1,\dots,v_k]+ D^k_\phi \mathcal{M}_o({\phi} , s,t,p_\s)[v_1,\dots,v_k]\Bigr)dt d\s\\
  &= \int_{S^1} \frac{1}{p_\s} \int_{\R}  D^k \mathcal{M}({\phi} , s,t, p_\s)dt d\s.
\end{align*}
Hence (\ref{eq:DkH-in-integ-variant}) is equivalent to (\ref{eq:DkH-in-integ}), and the proof of Proposition~\ref{prop:smooth-ovH} is finished.
\QED
\section{The linearization of  $H$ at constant functions}
\label{sec: linearized}

In this section, we compute the linearization of the operator $H$ at the constant function $\lambda \in \cU$, and we study key spectral properties of this linearization. We keep using the notation from the previous section. We need the following preliminary lemma on properties of the function 
\begin{equation}
  \label{eq:def-G}
G: \R \setminus \{0\} \to \R, \qquad G(t):=\int_{S^1}\frac{1}{(t^2+p^2_\s)^{\frac{3}{2}}} d\s,  
\end{equation}
which will appear in the expression for the linearization of $H$ at constant functions, see Lemma~\ref{lem:DERVATCOSNT}. 

\begin{Lemma}\label{LemmabehaveG}
  The function $G$ is even with
  \begin{equation}
    \label{eq:G-asymptotics}
    |G(t)| = O(|t|^{-3}) \qquad \text{as $|t| \to \infty$.}
  \end{equation}
Moreover, we have 
\begin{equation}\label{behaveG-eq}
G(t) =4 t^{-2}g(t) \qquad \text{for $t \in (0, +\infty)$}
\end{equation}
with a function $g\in C^{1}([0,\infty))$ satisfying  $g(t)>0$ for all $t \ge 0$.
\end{Lemma}
\proof
The evenness and (\ref{eq:G-asymptotics}) follow directly from the definition of $G$. By (\ref{eq:theta-sigma-formula}) we have 
\begin{align}\label{eq.GG}
G(t)&=4\int_{0}^{\frac{\pi}{2}} \frac{1}{(t^2+ 4\sin^2(\theta))^{\frac{3}{2}}} d\theta =4t^{-2}g(t),
\end{align}
where  $g\in C^\infty(0,\infty)$ is given by 
\begin{align}\label{eq.defgg}
g(t):=\int_{0}^{\frac{1}{t}} \frac{1}{(1+4r^2)^{\frac{3}{2}}} \frac{1}{\sqrt{1-(rt)^2}} d r \qquad\textrm{ for all $t>0$.}
\end{align}
We extend $g$ to $[0,\infty)$ by setting 
$$g(0):=\int_{0}^{+\infty} \frac{1}{(1+4r^2)^{\frac{3}{2}}} d r.$$
Then 
\begin{align}\label{eq: conVERG0}
g(t)-g(0)&=\int_{0}^{\frac{1}{t}} \frac{1}{(1+4r^2)^{\frac{3}{2}}}\biggl(\frac{1}{\sqrt{1-(rt)^2}}-1\biggl) dr+ \int_{\frac{1}{t}}^{+\infty} \frac{1}{(1+4r^2)^{\frac{3}{2}}} d r\nonumber\\
&=\frac{t^2}{2}\int_{0}^{\frac{1}{t}} \int_{0}^{1} \frac{1}{(1+4r^2)^{\frac{3}{2}}} \frac{r^2}{(1-\e(rt)^2)^{\frac{3}{2}}}d\e dr+ \int_{\frac{1}{t}}^{+\infty} \frac{1}{(1+4r^2)^{\frac{3}{2}}} d r\nonumber\\
&=\frac{t^2}{2}\int_{0}^{\frac{1}{2t}}  \int_{0}^{1} \frac{1}{(1+4r^2)^{\frac{3}{2}}} \frac{r^2}{(1-\e(rt)^2)^{\frac{3}{2}}}d\e dr\nonumber\\
& +\frac{1}{2}\int_{\frac{1}{2t}}^{\frac{1}{t}} \int_{0}^{1} \frac{1}{(1+4r^2)^{\frac{3}{2}}} \frac{(rt)^2}{\sqrt{1-\e(rt)^2}}d\e dr  
-\int_{\frac{1}{t}}^{+\infty} \frac{1}{(1+4r^2)^{\frac{3}{2}}} dr.
\end{align}
It is plain that
\begin{align}\label{eq: conVERG1}
\int_{\frac{1}{t}}^{+\infty} \frac{1}{(1+4r^2)^{\frac{3}{2}}} d r\leq C t^2\quad \textrm{for all}\quad  t\in \R.
\end{align}
In addition,
\begin{align}
\bigg|\int_{\frac{1}{2t}}^{\frac{1}{t}}  \int_{0}^{1} \frac{1}{(1+4r^2)^{\frac{3}{2}}} \frac{r^2}{(1-\e(rt)^2)^{\frac{3}{2}}}d\e dr\bigg| &\leq  \bigg|\int_{\frac{1}{2t}}^{\frac{1}{t}}  \int_{0}^{1}  \frac{r^{-1}}{(1-\e(rt)^2)^{\frac{3}{2}}}d\e dr\bigg| \nonumber\\
&\leq  \bigg|\int_{t}^{2t}  \int_{0}^{1}  \frac{\rho^2}{(\rho^2-\e t^2)^{\frac{3}{2}}}d\e d\rho\bigg|,\nonumber
\end{align}
and since 
$$\int_{0}^{1}  \frac{1}{(\rho^2-\e t^2)^{\frac{3}{2}}}d\e =\frac{2}{t^2}\biggl(\frac{1}{\rho}-\frac{1}{\sqrt{\rho^2-t^2}}\biggl),$$
it follows that
\begin{align}\label{eq: conVERG2}
\int_{t}^{2t}  \int_{0}^{1}  \frac{\rho^2}{(\rho^2-\e t^2)^{\frac{3}{2}}}d\e d\rho= \frac{2}{t^2} \biggl(\int_{t}^{2t}\rho d\rho-\int_{t}^{2t} \frac{\rho^2 d\rho}{\sqrt{\rho^2-t^2}}\biggl) \nonumber\\
= \frac{2}{t^2} \biggl(\int_{t}^{2t}\rho d\rho-t^2\int_{1}^{2} \frac{\rho^2 d\rho}{\sqrt{\rho^2-1}}\biggl).
\end{align}
Consequently,
\begin{align}
\bigg|\int_{\frac{1}{2t}}^{\frac{1}{t}} \int_{0}^{1} \frac{1}{(1+4r^2)^{\frac{3}{2}}} \frac{(rt)^2}{\sqrt{1-\e(rt)^2}}d\e dr\bigg|\leq  Ct^2.
\end{align}
For $r\leq \frac{1}{2t}$ and $\e\leq 1$, $1-\e(rt)^2\geq \frac{3}{4}$. Therefore for every $t\in (0, 1]$,
\begin{align}\label{eq: conVERG3}
\bigg|\int_{0}^{\frac{1}{2t}}  \int_{0}^{1} \frac{1}{(1+4r^2)^{\frac{3}{2}}} \frac{r^2}{(1-\e(rt)^2)^{\frac{3}{2}}}d\e dr\bigg| &\leq C \biggl (\int_{0}^{\frac{1}{2}} \frac{r^2}{(1+4r^2)^{\frac{3}{2}}} dr+\int_{\frac{1}{2}}^{\frac{1}{2t}}\frac{r^2}{(1+4r^2)^{\frac{3}{2}}} dr\biggl) \nonumber\\
& \leq C(1-\ln (t)). 
\end{align}
Combining  \eqref{eq: conVERG0}, \eqref{eq: conVERG1} \eqref{eq: conVERG2} and \eqref{eq: conVERG3}, we obtain
$$
| g(t)-g(0)|\leq C t^2(1-\ln (t)), \quad t\in (0, 1].
$$
Hence $g$ is differentiable at zero with 
\begin{align}\label{eq.defat0}
g'(0)=0.
\end{align}
We  now  check  the continuity of the derivative at zero.

By change of variable $\rho =rt$, we get from \eqref{eq.defgg},
$$
g(t)=t^2 \int_{0}^{1} \frac{1}{(t^2+4\rho^2)^{\frac{3}{2}}} \frac{1}{\sqrt{1-\rho^2}} d \rho=t^2 \int_{0}^{{\frac{\pi}{2}}} \frac{1}{(t^2+4\cos^2(\theta))^{\frac{3}{2}}}  d \theta,
$$
where we used $\rho=\cos(\theta)$ to get that last equality.  Hence 
\begin{align}\label{eq.dcontinudiffato2}
g(t)&=\frac{t^2}{2}  \int_0^{\pi} \frac{1}{(t^2+2+ 2\cos(\theta))^{\frac{3}{2}}} d\theta= \frac{t^2}{2} \int_0^{\pi} \frac{1}{\sqrt{a+ b\cos(\theta)})^{3}} d\theta,
\end{align}
where 
$a=t^2+2$ and $b=2$. 
Applying   \cite[Page 182, 3.]{I.S.I.M}, 
we find that 
\begin{align}\label{eq.dcontinudiffato1}
 \int_0^{\pi} \frac{1}{\sqrt{a+ b\cos(\theta)})^{3}} d\theta= \frac{2}{t^2 \sqrt{4+t^2}} E(\frac{\pi}{2}, r(t)), \quad t\ne 0
\end{align}
where $ r(t) =\frac{2}{ \sqrt{4+t^2}}$ and  $E(\cdot, \cdot)$ is the elliptic integral of second kind  defined (see \cite[Page 860, 3.]{I.S.I.M})  by
\begin{align}\label{eq.deE}
E(y, k)=\int_0^{y} \sqrt{1-k^2 \sin^2(\theta)}) d\theta=\int_0^{\sin(y)} \frac{\sqrt{1-k^2 x^2}}{\sqrt{1-x^2}} d x.
\end{align}
From  \eqref{eq.deE},  \eqref{eq.dcontinudiffato1} and  \eqref{eq.dcontinudiffato2}, 
\begin{align}\label{eq.dcontinudiffato}
g(t)=\frac{1}{\sqrt{4+t^2}} \int_0^{\frac{\pi}{2}} \sqrt{1-r^2(t) \sin^2(\theta) } d\theta.
\end{align}
We consider the function $h(t, \theta):=\sqrt{1-r^2(t) \sin^2(\theta) }$  defined for $\theta \in [0, \frac{\pi}{2}]$ and  $ t \in (0, \infty)$. Then
$$   \frac{\partial h}{\partial t}= \frac{ 4t \sin^2(\theta)}{(4+t^2)^2} \frac{1}{\sqrt{1-r^2(t) \sin^2(\theta) }}.$$
Since $\sin(\theta)\leq 1$, we have
\begin{equation}
  \label{eq:extra-eq-assumption-lebesgue}
\frac{t}{\sqrt{1-r^2(t) \sin^2(\theta)}} \leq \sqrt{4+t^2}  
\end{equation}
and
$$ \bigg| \frac{\partial h}{\partial t} \bigg| =  \frac{ 4t \sin^2(\theta)}{(4+t^2)^2} \frac{1}{\sqrt{1-r^2(t) \sin^2(\theta)}} \leq  \frac{ 4 \sin^2(\theta) }{(4+t^2)^{\frac{3}{2}} } \leq \frac{1}{2}.$$ By Lebesgue's theorem, we deduce with  \eqref{eq.dcontinudiffato} that $g$  is differentiable  and  for all $ t\in (0, \infty)$,
\begin{align}
  &g'(t)=\label{eq.diffg1}\\
  &-\frac{t}{(4+t^2)^{\frac{3}{2}}} \int_0^{\frac{\pi}{2}} \sqrt{1-r^2(t) \sin^2(\theta) } d\theta+ \frac{1}{(4+t^2)^{\frac{1}{2}} }\int_0^{\frac{\pi}{2}} \frac{4t\sin^2(\theta)}{(4+t^2)^{2}\sqrt{1-r^2(t) \sin^2(\theta) }} d\theta. \nonumber
\end{align}
Moreover,  it follows from \eqref{eq.defat0}, (\ref{eq:extra-eq-assumption-lebesgue}), \eqref{eq.diffg1} and Lebesgue's theorem that $g\in C^1([0,\infty))$. In addition, since $r^2(t)\leq 1$, \eqref{eq.dcontinudiffato} gives
\begin{align}\label{eq.dcontinudiffato23}
g(t)\geq \frac{1}{\sqrt{4+t^2}} \int_0^{\frac{\pi}{2}} \sqrt{1- \sin^2(\theta) } d\theta=\frac{1}{\sqrt{4+t^2}}> 0.
\end{align}
This completes the proof.
%
\QED

The next lemma yields a representation of the linearization of the operator $H$ at the constant function $\lambda \in \cU$ as a principal value integral operator.

\begin{Lemma}\label{lem:DERVATCOSNT}
At any constant function $\lambda \in \cU$, we have
\begin{align}
  &-\lambda  D_\phi H(\lambda)v(s)=\label{eq3diff}\\
  &PV \int_{\R}\{v(s)-v(s-\lambda\t )\}G(\t) d\t+\int_{\R}v(s-\lambda \t)F(\t)d \t-2\pi v(s),\nonumber
\end{align}
where $PV$ denotes the principal value integral, $G$ is defined in (\ref{eq:def-G}), and 
\begin{equation}
  \label{eq:def-F-}
F:=G_0-\frac{3}{4}G_1\quad \text{with}\quad G_0(\t):=\int_{S^1}\frac{p^2_\s}{(\t^2+p^2_\s)^{\frac{3}{2}}} d\s,\quad G_1(\t):=\int_{S^1}\frac{p^4_\s}{(\t^2+p^2_\s)^{\frac{5}{2}}} d\s.
\end{equation}
\end{Lemma}

\proof
We first observe that
\begin{align}\label{eqrFin L1}
G_0,G_1 \in L^1(\R)\qquad \text{and therefore also}\qquad F \in L^1(\R).
\end{align}
Indeed, by (\ref{eq:beta-function-eq}) we have 
\begin{align}
  \int_{\R}G_0(\t)d\t&=\int_{S^1}\int_{\R}\frac{p^2_\s}{(\t^2+p^2_\s)^{\frac{3}{2}}} d\s d\t=\int_{S^1}\frac{1}{p_\s}\int_{\R}\frac{1}{(\frac{\t^2}{p^2_\s}+1)^{\frac{3}{2}}} d\s d\t \nonumber\\
&=\int_{S^1}\int_{\R}\frac{1}{(t^2+1)^{\frac{3}{2}}} d\s dt=4\pi \int_{0}^\infty \frac{1}{(t^2+1)^{\frac{3}{2}}} d\s dt = 4 \pi \label{G_0-double-int}
\end{align}
and also 
\begin{equation}\label{eq:converg23}
\int_{\R}G_1(\t)d\t=\int_{S^1}\int_{\R}\frac{1}{(t^2+1)^{\frac{5}{2}}} d\s dt=\frac{8\pi}{3}.
\end{equation}
Hence (\ref{eqrFin L1}) follows.
Next we fix $\lambda>0$, regarded as a constant function in $\cU$. By \eqref{eq:Dk-T1T2} and Proposition~\ref{prop:smooth-ovH}, 
\begin{equation}
   \label{eq:DkH-in-integ-const} 
 \begin{aligned}
- D_\phi H(\lambda) v(s) = D_\phi \mathcal{H}(\lambda)v(s) &=\int_{S^1} \frac{1}{p_\s}  \int_{\R}  D_\phi \mathcal{M}(\lambda, s,t,p_\s)v \,dt d\s \\
 &+\frac{1}{2} \int_{S^1}\frac{1}{p_\s}   \int_{\R}D_\phi \mathcal{\ov{M}}(\lambda,s,t,p_\s)v\,dt d\s
\end{aligned}
\end{equation}
for $v\in C^{2,\alpha}_{p,e}(\R)$, $s \in \R$ with $\mathcal{H}$ given in (\ref{lem:DERVA1}). We now write 
\begin{align*}
\mathcal{M}(\phi, s,t, p) &=\mathcal{M}_1(\phi, s,t, p)\sqrt{1+(\phi')^2(s-\t p)},\\
\mathcal{\ov{M}}(\phi, s,t, p)& :=\mathcal{M}_2(\phi, s,t, p)\sqrt{1+(\phi')^2(s-t p)}
\end{align*}
where
\begin{align*}
\mathcal{M}_1(\phi, s,t, p):&=t \L_1(\phi, s,t, p)\mathcal{K}(\phi, s, t, p)= t \L_1(\phi, s,t, p)\phi(s-tp) \mathcal{\widetilde K}(\phi, s, t, p) \nonumber\\
\mathcal{M}_2(\phi, s,t, p):&=\phi(s-t p)\mathcal{K}(\phi, s, t, p)=\phi(s-t p)^2\mathcal{\widetilde K}(\phi, s, t, p) 
\end{align*}
with
$$
\mathcal{\widetilde K}({u} , s,t,p) = \frac{1}{\{t^2+ t^2 \L^2_0(u,s,t p)+ u(s)u(s-t p)\}^{\frac{3}{2}}}.
$$
Since $\L_1(\lambda, s,t, p)=0$, we have
\begin{align*}
&D_\phi \mathcal{M}_1(\lambda, s,t, p)v=\lambda t \L_1(v, s,t, p)\mathcal{\widetilde K}(\lambda, s, t, p)\\
&=t\biggl(\frac{v(s)-v(s-tp)}{pt}-v'(s)\biggl)\frac{\lambda}{(t^2+\lambda^2)^{\frac{3}{2}}}=\frac{t}{\lambda^2} \biggl(\frac{v(s)-v(s-tp)}{pt}-v'(s)\biggl)\frac{p^3}{(\frac{t^2p^2}{\lambda^2} +p^2)^{\frac{3}{2}}}.
\end{align*}
Hence, making the change of variable $\t=\frac{tp_\s}{\lambda}$, we get
\begin{align}
&\int_{S^1} \frac{1}{p_\s}  \int_{\R} D_\phi \mathcal{M}_1(\lambda, s,t, p_\s)v dt d\s \nonumber=\frac{1}{\lambda}\int_{S^1}  \int_{\R} \frac{v(s)-v(s-\lambda\t )-\lambda\t v'(s)} {(\t^2+p^2_\s)^{\frac{3}{2}}}  d\t d\s\nonumber\\
  &=\frac{1}{\lambda} \lim_{\e\rightarrow 0}\biggl(\int_{S^1} \int_{|\t|\geq\e} \frac{v(s)-v(s-\lambda\t )-\lambda\t v'(s)} {(\t^2+p^2_\s)^{\frac{3}{2}}}  d\t d\s\biggr),\nonumber
\end{align}
where we used the Lebesgue's theorem together with Lemma~\ref{LemmabehaveG} and the estimate $|v(s)-v(s-\lambda\t )-\lambda\t v'(s)| \le \min(1,|\lambda\t|^{1+\alpha})$. Consequently, by Fubini's theorem,
\begin{align}\label{eq:diff1}
&\lambda \int_{S^1} \frac{1}{p_\s}  \int_{\R} D_\phi \mathcal{M}_1(\lambda, s,t, p_\s)v dt d\s   \\
  &=\lim_{\e\rightarrow 0} \int_{|\t|\geq\e}\Bigl( v(s)-v(s-\lambda\t )-\lambda\t v'(s)\Bigr)G(\tau)d\t \nonumber\\
  &=\lim_{\e\rightarrow 0} \int_{|\t|\geq\e}\bigl( v(s)-v(s-\lambda\t )\bigr)G(\tau)d\t = PV \int_{\R}\{v(s)-v(s-\lambda\t )\}G(\t) d\t \nonumber
\end{align}
with $G$ given in (\ref{eq:def-F-}). We also have $$D_\phi \mathcal{\widetilde K}(\lambda, s,t, p_\s)v(s)=-\frac{3\lambda}{2}\frac{v(s)+v(s-tp_\s)}{(t^2+\lambda^2)^{\frac{5}{2}}}$$ and 
\begin{align*}
D_\phi \mathcal{M}_2(\lambda, s,t, p_\s)v&=\lambda\biggl(2\mathcal{\widetilde K}(\lambda, s,t, p_\s)v(s-tp_\s)+\lambda D_\phi \mathcal{\widetilde K}(\lambda, s,t, p_\s)v\biggl)\nonumber\\
&=\lambda \biggl(2\frac{v(s-tp_\s)}{(t^2+\lambda^2)^{\frac{3}{2}}}-\frac{3\lambda^2}{2}\frac{v(s)+v(s-tp_\s)}{(t^2+\lambda^2)^{\frac{5}{2}}}\biggl)\\
&=\frac{1}{\lambda^2}p_\s^3 \biggl(\frac{2v(s-tp_\s)}{(\frac{t^2p_\s^2}{\lambda^2} +p_\s^2)^{\frac{3}{2}}}-\frac{3}{2}\frac{p_\s^2(v(s)+v(s-tp_\s))}{(\frac{t^2p_\s^2}{\lambda^2} +p_\sigma^2)^{\frac{5}{2}}}\biggl),
\end{align*}
yielding, again by the change of variable $\t=\frac{tp_\s}{\lambda}$,
\begin{align}
&\frac{\lambda}{2} \int_{S^1}  \int_{\R}  D_\phi \mathcal{M}_2(\lambda, s,t, p_\sigma)v \,dt \,d \s   \label{eq:diff2}\\
&= \int_{S^1}  \int_{\R} \biggl(\frac{p^2_\s v(s-\lambda\t )}{(\t^2+p^2_\s)^{\frac{3}{2}}} -\frac{3}{4}\frac{p^4_\s (v(s)+v(s-\lambda\t))}{(\t^2+p^2_\s)^{\frac{5}{2}}}\biggl) d\t d\s\nonumber\\
&= \int_{\R}\Bigl(G_0(\t) v(s-\lambda \t)-\frac{3}{4}G_1(\t)\bigl(v(s)+ v(s-\lambda \t)\bigr)\Bigr)d \t\nonumber\\
&= \Bigl(\int_{\R} v(s-\lambda \t)F(\t)d\t  -\frac{3}{4}v(s) \int_{\R} G_1(\t)d \t \Bigr)= \int_{\R} v(s-\lambda \t)F(\t)d\t - 2 \pi v(s) \nonumber
\end{align}
with $G_0$ and $G_1$ given in (\ref{eq:def-F-}). Here we used \eqref{eq:converg23} and Fubini's theorem together with the finiteness of the integrals in (\ref{G_0-double-int}) and (\ref{eq:converg23}). Combining (\ref{eq:DkH-in-integ-const}) with \eqref{eq:diff1} and \eqref{eq:diff2}, we deduce \eqref{eq3diff}.
\QED

Next, we study the eigenvalues of the linearized operator
\begin{equation}
  \label{eq:def-L-lambda}
L_\lambda:=-\lambda  D_\phi H(\lambda)
\end{equation}
computed in Lemma \ref{lem:DERVATCOSNT}. As we shall see, the eigenvalues are expressed in term of  modified Bessel functions  $I_{\nu}$  and $K_{\nu}$.  We shall use various identities, inequalities and asymptotics involving these functions, and we collect them in the appendix of this paper for the convenience of the reader.

\begin{Lemma}\label{lem:eigenval}
Let   $\l>0$. The functions  $e_k \in C^{2,\alpha}_{p,e}(\R)$ given by
\begin{equation} \label{eq:def-e-k}
s \mapsto e_k(s )=\cos(k s ), \quad k \in \N \cup  \{0\}
\end{equation}
are the eigenfunctions of the operator $L_\lambda$ given (\ref{eq:def-L-lambda}). More precisely,
\be \label{eq:L-lam-ek-eq-nuk-ek}
L_\l e_k=V(\l k) e_k,
\ee
with
\begin{align} \label{eq:lkhk}
V(\rho)= 4\pi\rho I_1(\rho)\biggl(\rho K_1(\rho)-K_0(\rho)\biggl) \qquad \text{for $\rho>0$.}
\end{align}
Here $I_\nu$ and $K_\nu$ are the Bessel functions defined in the appendix in \eqref{Bes1} and \eqref{Bes2} for  $\nu\geq 0$. 
\end{Lemma}

\proof
To get \eqref{eq:lkhk}, we recall that
$$
L_\l e_k(s)= PV \int_{\R}\{e_k(s)-e_k(s-\lambda\t )\}G(\t) d\t+\int_{\R}e_k(s-\lambda \t)F(\t)d \t-2\pi e_k(s)
$$
by (\ref{eq3diff}). Using the identity
\begin{equation}\label{eq: rapeel}
e_k(s-\lambda\t )=\cos\bigl(k (s-
                   \lambda \t)\bigr) =\cos(ks) \cos(k \lambda \t)+\sin(ks)\sin(k \lambda \t),
\end{equation} 
the evenness of the  functions $G$  and $F$ and the oddness of the sine
function, we get 
$$
L_\lambda  e_k(s)= \biggl(\int_{\R}\bigl(1-e_k(\lambda \t)\bigr)G(\t) d\t+\int_{\R} e_k(\lambda \t))F(\tau)d \tau-2\pi\biggl) e_k(s), 
$$
and this proves \eqref{eq:L-lam-ek-eq-nuk-ek} with
\begin{equation}\label{eq:valprobiendef}
V( \rho):=\int_{\R}\{1-\cos(\rho t) \}G(t) dt+ \int_{\R} \cos(\rho t)F(t)d t-2\pi.
\end{equation}
It thus remains to prove that (\ref{eq:lkhk}) holds with this definition of $V$.
For this we write 
$$1-\cos(\rho  t)=\rho  \int^1_{0} t\sin( r \rho t)\textrm{d}r, $$
which gives, by (\ref{eq:theta-sigma-formula}) and the integral identity \eqref{eq:integral-identity-1-special-case}, 
\begin{align}
  &\int_{\R}\{1-\cos(\rho  t)\}G(t) dt=\rho \int_{\R} \int_{0}^{1} \int_{S^1}
    \frac{t\sin(r \rho  t)}{(t^2+p_\sigma^2)^{\frac{3}{2}}}d\s  dt  d r \nonumber\\
  &=4 \rho \int_{\R}  \int_{0}^{1} \int_{0}^{\frac{\pi}{2}}\biggl(  \frac{t\sin( r \rho  t)}{(t^2+4\sin^2(\th))^{\frac{3}{2}}}  dt \biggl)d\th  dt  d r \nonumber\\
  &=8 \rho \int_{0}^{\frac{\pi}{2}} \int_{0}^{1} \biggl(\int_{0}^{\infty}  \frac{t\sin( r \rho  t)}{(t^2+(2\sin \th)^2)^{\frac{3}{2}}}  dt \biggl)d r d\th \nonumber \\
  &= 8 \rho^2 \int_{0}^{\frac{\pi}{2}} \int_{0}^{1}  r  K_0(2\rho r \sin \th )drd\th. \nonumber 
\end{align}
Since, by (\ref{eq:formderibess}) and (\ref{eqasymp}), we also have
\begin{align*}
  \int_{0}^{1}  r  K_0(2\rho r \sin \th )dr &= (2\rho  \sin \th)^{-2} \int_0^{2\rho \sin \th} r K_0(r)\,dr\\
  &=
\frac{1}{(2\rho)^2\sin^2(\th)}\Bigl(1-2 \rho \sin(\th)K_1(2\rho \sin(\th))\Bigl),
\end{align*}
we conclude that 
\begin{align}
  \int_{\R}\{1-\cos(\rho  t)\}G(t) dt&=2 \int_{0}^{\frac{\pi}{2}}\frac{1}{\sin^2(\th)}\Bigl(1-2 \rho \sin(\th)K_1(2\rho \sin(\th))\Bigl)d\th. \nonumber\\
&= -2 \int_{0}^{\frac{\pi}{2}}\cot(\th) \frac{d}{d\th} \Bigl(1-2 \rho \sin(\th)K_1(2\rho \sin(\th))\Bigl)d\th \nonumber\\
&= 8 \rho^2 \int_{0}^{\frac{\pi}{2}}\cos^2(\th) K_0(2\rho \sin(\th))d\th  \label{eq:valbaseel}
\end{align}
Here we used integration by parts and (\ref{eq:formderibess}), noting in particular that $1- x K_1(x) =O(x^2)$ as $x \to 0^+$ as a consequence of (\ref{eq:formderibess}) and (\ref{eqasymp}).

Next we compute, using (\ref{eq:theta-sigma-formula}), Fubini's theorem and (\ref{eq:integral-identity-2-special-case}), 
\begin{align}
&  \int_{\R} \cos(\rho t) G_0(t) dt= \int_{0}^{\infty}  \int_{S^1}
                 \frac{\cos(\rho t)p_\s^2}{(t^2+ p_\s^2)^{\frac{3}{2}}}d\s  dt \label{eq:valbaseel1}\\
  &= 32 \int_{0}^{\frac{\pi}{2}} \int_{0}^{\infty}  \frac{\cos(\rho t) \sin^2\th}{(t^2+4\sin^2 \th))^{\frac{3}{2}}}  dt d\th
=16 \rho \int_{0}^{\frac{\pi}{2}} \sin(\th)  K_1(2\rho \sin(\th))d\th. \nonumber
\end{align}
Proceeding in the same manner, using now (\ref{eq:theta-sigma-formula}), Fubini's theorem and (\ref{eq:integral-identity-2-special-case}), we get
\begin{align}\label{eq:valbaseel2}
\int_{\R}  \cos(\rho t)G_1(t) dt&=\frac{32}{3}\rho^2  \int_{0}^{\frac{\pi}{2}}  \sin^2(\th)  K_2(2\rho \sin(\th))d\th\nonumber\\
&=\frac{32}{3}\rho  \int_{0}^{\frac{\pi}{2}}  \sin(\th)  K_1(2\rho \sin(\th))d\th+\frac{32}{3}\rho^2  \int_{0}^{\frac{\pi}{2}}\sin^2(\th)  K_0(2\rho \sin(\th))d\th,
\end{align}
where we have used \eqref{eq:valK2} to get the last equality. Since $F=G_0-3/4G_1$, it follows  from  \eqref{eq:valbaseel1} and  \eqref{eq:valbaseel2} that
\begin{align} \label{eq:calFconvole}
\int_{\R} e_k(\l t)F(t)d t&=8\rho \int_{0}^{\frac{\pi}{2}}  \sin(\th)  K_1(2\rho  \sin(\th))d\th-8\rho^2 \int_{0}^{\frac{\pi}{2}} \sin^2(\th)  K_0(2\rho \sin(\th))d\th\nonumber\\
&=V_2(\rho)-V_3(\rho),
\end{align}
with
$$
V_2(\rho):= 8\rho \int_{0}^{\frac{\pi}{2}}  \sin(\th)  K_1(2\rho  \sin(\th))d\th\quad \textrm{and}\quad 
V_3(\rho):=  8\rho^2 \int_{0}^{\frac{\pi}{2}}  \sin^2(\th)  K_0(2\rho \sin(\th))d\th.
$$
We also set 
$$V_1(\rho):= 8 \rho^2 \int_{0}^{\frac{\pi}{2}}\cos(\th) K_0(2\rho \sin(\th))d\th.$$
Then \eqref{eq:valprobiendef}, \eqref{eq:valbaseel} and \eqref{eq:calFconvole}   yield
\begin{align} \label{eq:valbaE2}
V(\rho)&=V_1(\rho)+  V_2(\rho)-V_3(\rho)-2\pi.
\end{align}

We now derive further identities for the  functions $V_1$, $V_2$ and  $V_3$.

Using the identity $2 \sin^2 \th = 1 - \cos(2\th)$ and (\ref{eq:final-integral-identity}), we find that 
\begin{align}\label{compV1}
V_1(\rho)&=8 \rho^2 \int_{0}^{\frac{\pi}{2}}\cos^2(\th) K_0(2\rho \sin(\th))d\th = 8\rho^2 \int_{0}^{\frac{\pi}{2}} \sin^2(\th)K_0(2\rho \cos (\th) )d\th \nonumber\\
&=4\rho^2\biggl(\int_{0}^{\frac{\pi}{2}} K_0(2\rho \cos (\th) )d\th-\int_{0}^{\frac{\pi}{2}} K_0(2\rho \cos (\th) ) \cos (2\th)d\th\biggl)\nonumber\\
&=2\pi \rho^2 \biggl(I_0(\rho)K_0(\rho)+I_1(\rho)K_1(\rho)\biggl).
\end{align}
We also have, due to the identity $2 \cos^2 \th = 1 + \cos(2\th)$ and (\ref{eq:final-integral-identity}),
\begin{align}\label{eq:valV3}
V_3(\rho)&= 8\rho^2 \int_{0}^{\frac{\pi}{2}}  \sin^2(\th)  K_0(2\rho \sin(\th))d\th= 8\rho^2 \int_{0}^{\frac{\pi}{2}} \cos^2(\th)K_0(2\rho \cos (\th) )d\th \nonumber\\
&= 4\rho^2\biggl(\int_{0}^{\frac{\pi}{2}} K_0(2\rho \cos (\th) )d\th + \int_{0}^{\frac{\pi}{2}} K_0(2\rho \cos (\th) )\cos(2\th)d\th\biggr)\nonumber\\
&= 2\pi \rho^2 \biggl( I_0(\rho)K_0(\rho)-I_1(\rho)K_1(\rho)\biggl),
\end{align}
where we used again (\ref{eq:final-integral-identity}). Finally, we compute $V_2$. By Lebesgue's theorem, (\ref{eqasymp}) and \eqref{eqlimi11}, 
\begin{align} \label{eq:limV22}
\lim \limits_{\rho \to 0^+}  V_2(\rho) = 2 \pi.
\end{align}
Moreover, $V_2$ is  
differentiable on $\R_{+}$, and from \eqref{eq:formderibess} and \eqref{eq:valV3} we deduce that 
\begin{align} \label{eq:defV2}
V_2'(\rho) &=-16\rho  \int_{0}^{\frac{\pi}{2}} \sin^2(\th)K_0(2\rho \sin (\th) )d\th=-\frac{2}{\rho}V_3(\rho)\nonumber\\
&=-4 \pi\rho\biggl(I_0(\rho)K_0(\rho)-I_1(\rho)K_1(\rho)\bigg).
\end{align} 
Setting now $W(\rho):=4\pi \rho I_0(\rho)K_1(\rho)$, we see that $\lim \limits_{\rho \to 0^+}  W(\rho) = 4\pi$ by (\ref{eqasymp}) and since $I_0(0)=1$, whereas 
$$W'(\rho)=-4\pi\rho\biggl(I_0(\rho)K_0(\rho)-I_1(\rho)K_1(\rho)\bigg)=V_2'(\rho).$$
by \eqref{relateI} and \eqref{relateK}. This together with \eqref{eq:limV22} implies that 
\begin{align}\label{eq:val4V2}
V_2(\rho)&=W(\rho)-2\pi = 4\pi\rho I_0(\rho)K_1(\rho)-2\pi.
\end{align}
Inserting the identities \eqref{compV1}, \eqref{eq:valV3}, \eqref{eq:val4V2} in \eqref{eq:valbaE2}, we obtain
\begin{align}\label{eq:ExpressssV}
V(\rho)= 4\pi\rho^2 I_1(\rho)K_1(\rho) +4\pi\rho I_0(\rho)K_1(\rho)-4\pi.
\end{align}
Finally, we observe that 
$$
4\pi\rho I_0(\rho)K_1(\rho)=4\pi-4\pi\rho I_1(\rho)K_0(\rho)
$$
by the Wronskian identity \eqref{relawrosk}, and we get \eqref{eq:lkhk}.
\QED 

\begin{Lemma}\label{asymtotic}
The function $V$ defined in \eqref{eq:valprobiendef}  has the following asymptotics:
\begin{enumerate}
\item[(i)] $\lim \limits_{\rho \to 0^+}  V(\rho) = 0$. 
\item[(ii)]  $\lim \limits_{\rho \to +\infty} \frac{V(\rho)}{\rho}=2\pi.$ In particular, $\lim \limits_{\rho \to + \infty} V(\rho)= +\infty$.
\end{enumerate}
\end{Lemma}
\proof
Recalling  \eqref{eq:lkhk}, we find that 
\begin{align} \label{eq:defVAL2}
V(\rho)= 4\pi  \rho I_1(\rho)\biggl(\rho K_1(\rho)- K_0(\rho)\biggl)=4\pi I_1(\rho)\biggl(\rho^2 K_1(\rho)-\rho K_0(\rho)\biggl) \to 0 \qquad \text{as $\rho \to 0^+$}
\end{align}
as a consequence of \eqref{eqasymp} and the fact that $I_1(\rho) \to I_1(0)=0$ as $\rho \to 0^+$.

To prove (ii), we use that $K_1>K_0$ (see \eqref{eqCompK0K1}) and the bound \eqref{eq:defVAL2} to obtain that 
\begin{equation} \label{eq:defVALinfboud}
4\pi(\rho-1)\rho I_1(\rho)K_0(\rho)\le V(\rho)\le 4\pi  \rho^2 I_1(\rho) K_1(\rho) \qquad \text{for $\rho>0$.}
\end{equation}
Moreover, by the asymptotics given in (\ref{eq:asyI1}), we have
\begin{align}\label{eq:limIK}
 \lim  \limits_{\rho \to +\infty } \rho I_1(\rho)K_0(\rho)=\lim  \limits_{\rho \to +\infty } \rho I_1(\rho)K_1(\rho)=\frac{1}{2}.
\end{align}
Combining (\ref{eq:defVALinfboud}) and (\ref{eq:limIK}) readily yields $\lim \limits_{\rho \to +\infty} \frac{V(\rho)}{\rho}=2\pi$, as claimed in (ii).
\QED

The next lemma is of key importance for the application of the Crandall-Rabinowitz theorem in Section~\ref{ss:Existt-d-per-excepti} below.

\begin{Lemma}\label{derv222}
Let  $V$  be the function defined in \eqref{eq:lkhk}. Then for all $\rho \in (0, +\infty)$,
\begin{equation}\label{eq:derVVnew1}
V' (\rho)=4\pi\rho^2\biggl(I_0(\rho)K_1(\rho)-I_1(\rho)K_0(\rho)\biggl)-4\pi\rho\biggl(I_0(\rho)K_0(\rho)-I_1(\rho)K_1(\rho)\biggl).
\end{equation}
Moreover,
\begin{equation}
\label{derv222-second-claim}
\text{there exists a unique $\lambda_*> 0$ satisfying $V(\lambda_*)=0$, and we have $V'(\lambda_*)>0.$}
\end{equation}
\end{Lemma}

\proof
By a straighforward computation using \eqref{eq:lkhk}, (\ref{relateI}) and \eqref{relateK}, we get  \eqref{eq:derVVnew1}.

It thus remains to prove (\ref{derv222-second-claim}). We start by observing  from the right inequality in \eqref{eq:compK1K0} 
that
\begin{equation}
  \label{eq:extra-inequ-1}
\rho K_1(\rho)< \frac{1+2\rho}{2} K_0(\rho) \qquad \text{for all $\rho >0$.} 
\end{equation}
Inserting this in \eqref{eq:lkhk},
we obtain
\begin{align}\label{eq:existzeroV}
V(\rho)&< 2\pi\rho I_1(\rho)(2\rho-1)K_0(\rho).
\end{align}
Hence,
\begin{align}\label{eq:NEGATIVDER}
V(\rho)<0\quad \textrm{for all}\quad \rho \in (0, \frac{1}{2}).
\end{align}
As a consequence of Lemma~\ref{asymtotic} and \eqref{eq:NEGATIVDER}, there exists a number $\rho_*>\frac{1}{2}$ such that $V(\rho_*)=0$.

Hence, to show (\ref{derv222-second-claim}), it now suffices to show that  $V(\rho)=0$ implies $V'(\rho)>0$ for $\rho >\frac{1}{2}$. Let $\rho>\frac{1}{2}$ such that $V(\rho)=0$. Then from  \eqref{eq:lkhk}
\begin{align}\label{eq:vvanis}
\rho K_1(\rho)=K_0(\rho), \quad \rho>\frac{1}{2}.
\end{align}
This with (\ref{eq:extra-inequ-1}) and \eqref{eq:compK1K0} yields the inequality
$$
  \rho K_1(\rho)=K_0(\rho)< \frac{1+4\rho}{3+4\rho}K_1(\rho),
$$
which allows  to see that any zero  of $V$ is confined in the interval
\begin{align}\label{eq:rangenBif23}
\frac{1}{2} <\rho < \frac{1+\sqrt{17}}{8}.
\end{align}
We now prove that
\begin{equation}
  \label{eq:necess-cond-uniqueness}
\text{$V'(\rho) >0$ for all $\rho \in \bigl(\frac{1}{2}, \frac{1+\sqrt{17}}{8}\bigr)$ with $V(\rho)=0$.}
\end{equation}
After inserting \eqref{eq:vvanis} in \eqref{eq:derVVnew1},
\begin{align}\label{eq:dercomparrr}
\frac{ V' (\rho)}{4\pi}&=\rho^2\biggl(I_0(\rho)K_1(\rho)-I_1(\rho)K_0(\rho)\biggl)-\rho I_0(\rho)K_0(\rho)+ I_1(\rho)K_0(\rho).
\end{align}
Set  
\begin{equation}\label{eq:defAB}
B(\rho):=I_0(\rho)K_1(\rho)-I_1(\rho)K_0(\rho)\quad \text{and}\quad 
D(\rho):=I_0(\rho)K_0(\rho) 
\end{equation}
so that 
\begin{align}\label{eq:dercomparrr33}
\frac{ V' (\rho)}{4\pi}&=\rho^2 B(\rho)-\rho D(\rho)+I_1(\rho)K_0(\rho).
\end{align}
By computation using \eqref{eq:compI1I00} and \eqref{eq:compK1K0}, 
\begin{align}
&\rho B(\rho)> \frac{\rho(6+7\rho-4\rho^2)}{2(1+4\rho)} D(\rho)\quad\textrm{and}\quad  I_1(\rho)K_0(\rho)> \frac{\rho}{2+\rho}D(\rho)\label{eq:compBD2}
\end{align}
and \eqref{eq:dercomparrr33} yields, after computation,
\begin{align}\label{eq:dercomparrr2}
\frac{ V' (\rho)}{4\pi}&=\rho^2 B(\rho)-\rho D(\rho)+I_1(\rho)K_0(\rho)> \biggl(\rho B(\rho)-\frac{1+\rho}{2+\rho}\biggl)\rho D(\rho)\nonumber\\
&>\Bigl(\dfrac{2 \rho +12 \rho^2-\rho^3-4 \rho^4-2}{2(1+4 \rho)(2+\rho)}\Bigl)\rho D(\rho).
\end{align}
Finally, a simple comparison using the bounds in  \eqref{eq:rangenBif23} shows that the function $\rho\mapsto 2 \rho +12 \rho^2-\rho^3-4 \rho^4-2$ is positive on the intervall $\bigl(\frac{1}{2}, \frac{1+\sqrt{17}}{8}\bigr)$. This shows (\ref{eq:necess-cond-uniqueness}), and from (\ref{eq:necess-cond-uniqueness}) it follows that $V$ has at most one zero in the interval $\bigl(\frac{1}{2}, \frac{1+\sqrt{17}}{8}\bigr)$. This ends the proof.
\QED

\section{Existence of exceptional domains}\label{ss:Existt-d-per-excepti}

In the following, we consider the fractional Sobolev spaces
\begin{equation}
  \label{eq:def-hpe}
H^{\s}_{p,e} := \Bigl \{v \in H^{\s}_{loc}(\R) \::\: \text{
$v$ even and $2\pi$-periodic}\Bigl \}
\end{equation}
for $\s\geq 0$, and we put $L^2_{p,e}:= H^{0}_{p,e}$. Note that
$L^2_{p,e}$ is a Hilbert space with scalar product
$$
(u,v) \mapsto \langle u,v \rangle_{L^2_{p,e}} :=
\int_{-\pi}^{\pi} u(t)v(t)\,dt \qquad \text{for $u,v \in
L^2_{p,e}$.}
$$
We denote the induced norm by $\|\cdot\|_{L^2_{p,e}}$. Consider the functions $e_k$, $k \in \N$ defined in (\ref{eq:def-e-k}). Since
$\|e_k\|_{L^2_{p,e}}=\sqrt{\pi}$, the set
$\{\frac{e_k}{\sqrt{\pi}},\:\,k\in \N\}$ forms  a
complete orthonormal basis of $L^2_{p,e}$.
Moreover, $H^\s_{p,e}
\subset L^2_{p,e}$ is characterized as the subspace of all functions
$v \in L^2_{p,e}$ such that
\begin{align}\label{eqchatacte}
\sum_{k \in \N } (1+k^2)^{\s} \langle v, e_k
\rangle_{L^2_{p,e}} ^2 < \infty.
\end{align}
Therefore, $H^\s_{p,e}$ is also a Hilbert space with scalar product
\begin{equation}
  \label{eq:scp-hj}
(u,v) \mapsto  \sum_{k \in \N } (1+k^2)^{\s} \langle u, e_k
\rangle_{L^2_{p,e}}  \langle v, e_k \rangle_{L^2_{p,e}}  \qquad
\text{for $u,v \in H^\s_{p,e}$.}
\end{equation}

Defining 
\begin{equation}
  \label{eq:def-Vell}
W_k:=\textrm{span}\left\{ e_k \right\}
\,\subset \,\bigcap_{j \in \N} H^j_{p,e}
\end{equation}
for $k \in  \N$, we see, by Lemma  \ref{lem:eigenval}, that the spaces $W_k$ are precisely the eigenspaces of the operator $L_\l$ in given in \eqref{eq:def-L-lambda} corresponding to the eigenvalues $V(\l k)$, i.e., we have  
 \be\label{eq:-Lv-Fpurier-slab}
 L_\l v =V(\l k) v \qquad \textrm{ for every $v\in W_k$}.
\ee
We also consider their orthogonal complements in $L^2_{p,e}$, given by
$$
W_k^\perp:=\left\{ w\in L^2_{p,e} \,:\,
\int_{-\pi}^{\pi} \cos(k s) w(s)\,ds=0 \right\}.
$$
For fixed $\alpha \in (0,1)$  as before, we now define the spaces 
$$
X:=\biggl\{  \phi: \mathbb{R} \rightarrow \mathbb{R},\quad  \phi\in  C^{2,\alpha}(\mathbb{R}) \textrm{ is even and $2\pi$-periodic}\quad \textrm{and} \quad \langle \phi, 1  \rangle_{L^2_{p,e}}=0\biggl\},
$$
and
$$
Y:=\biggl\{  \phi: \mathbb{R} \rightarrow \mathbb{R},\quad  \phi\in  C^{1,\alpha}(\mathbb{R}) \textrm{ is even and $2\pi$-periodic} \quad  \textrm{and} \quad \langle \phi, 1  \rangle_{L^2_{p,e}}=0 \biggl\}.
$$
In order  to prove our main result, we shall apply the 
Crandall-Rabinowitz bifurcation theorem to the smooth nonlinear map
\begin{equation}\label{CR-map}
\Phi: \cD_\Phi \subset \R \times X  \to Y, \qquad \qquad \Phi(\l,\varphi): = H(\lambda + \phi)+ 2\pi 
\end{equation}
defined on the open set
\be \label{eq:DPhi} \cD_\Phi:= \left\{(\l,\varphi)\::\: \l>0,\:
\varphi \in X,\, \min_{\R} \varphi>-\l\right\} \;\subset
\; \R \times X.  \ee 
As we have proved in Proposition~\ref{prop:smooth-ovH}, the map $\Phi$ is smooth, and by \eqref{HatCos} we have
$$
\Phi (\l,0)=0 \qquad \text{for every $\l>0$.}
$$

To apply the Crandall-Rabinowitz bifurcation theorem, we need to verify some assumptions regarding the linearized operators $D_\varphi \Phi(\l,0): X \to Y$, $\lambda>0$, which are given as
\begin{equation}
  \label{eq:relationship-D-phi-L-lambda}
D_\varphi \Phi(\l,0) = -\frac{1}{\lambda}L_\lambda
\end{equation}
with $L_{\lambda}$ defined in (\ref{eq:def-L-lambda}). To proceed, we need some preliminaries. In particular,  regularity results for the following  equation 
\begin{align} \label{condighallapla}
PV \int_{\R}\{v(s)-v(\t )\}G(s-\t) d\t +\l v= f(s), \quad s\in \R,
\end{align}
will be used for the proof of Proposition \eqref{prop:All-HYPcRANDAL} below. 
From  Lemma \ref{LemmabehaveG}, we have $|G(t)| = O(|t|^{-3}) \qquad \text{as $|t| \to \infty$.}$ Furthermore since  $G(-\t)=G(\t)$ for all $\t>0$, weak solutions $v\in H^{1/2}_{p,e} $ of  \eqref{condighallapla}  are characterised by 
\begin{align}\label{weaksol}
  \frac{1}{2}\int_{\R} \int_{\R}\{v(s)-v(\t )\}\{\psi(s)-\psi(\t )\}G(s-\t) ds
  d\t =
  \int_{\R} (f -\l v)\psi\, d\t \quad \text{for all $\psi \in C^\infty_c(\R)$.}
\end{align}

\begin{Lemma} \label{LemmTworegul}
Let $\l\in \R$, $\s\in (0,1)$ and  $f\in L^{2}_{p,e}$.  Let $v\in H^{1/2}_{p,e} $ be a weak solution of 
\begin{align} \label{condighallapla1}
PV \int_{\R}\{v(s)-v(\t )\}G(s-\t) d\t +\l v= f(s), \quad s\in \R
\end{align}
in the sense that \eqref{weaksol} holds.
Then, if $f\in C^{m,\s}(\R)$ for some $m\in \{0,1\}$,  we have  $v\in C^{m+1,\s}(\R)$, and (\ref{condighallapla1}) holds pointwisely.
\end{Lemma}

\begin{proof}
We have that $G(-\t)=G(\t)$ for all $\t>0$. In addition, by Lemma \ref{LemmabehaveG}, we have $ t^2G(  t)=4 g(t)$, for some function $g\in C^{1}([0,\infty)$ with $g(t)>0$ for all $t \ge 0$. Hence for every $R>0$ we have $g\in C^{\g}([0,R])$ for all $\g\in (\s,1)$ and $\inf \limits_{[0,R]}g>0$.  We can thus apply  \cite{fall2020regularity}[Theorem 1.3-$(i)$]  to get 
$$
\|v\|_{C^{\s}([-R/2,R/2])}\leq   C(R,\s,\g,\l) \left(\|v\|_{L^2(-R,R)}+\int_{\R}\frac{|v(t)|}{1+ |t|^{2}}\,dt+ \|f\|_{L^{\infty}(\R)} \right).
$$
Therefore since $v$ is periodic, we get
\begin{equation}
  \label{eq:C1-a-est-01}
\|v\|_{C^{\s}(\R)}\leq C(\|v\|_{L^2_{p,e}}+  \|f\|_{L^{\infty}(\R)}  )  
\end{equation}
Now, by
 \cite{fall2020regularity}[Theorem 1.4-$(iv)$] and using again that $v$ is periodic, we obtain
\begin{equation}
  \label{eq:C1-a-est-02}
\|v\|_{C^{1,\s}(\R)}\leq   C \left(\|v\|_{L^{\infty}(\R)}+ \|f-\l v \|_{C^{\s}(\R)} \right).
\end{equation}
Combining (\ref{eq:C1-a-est-01}) and (\ref{eq:C1-a-est-02}), we deduce that
\be\label{eq:C1-a-est}
\|v\|_{C^{1,\s}(\R)}\leq \left( \|v\|_{L^2_{p,e}}+  \|f\|_{C^{\s}(\R)} \right) .
\ee
Suppose now that $f\in C^{1,\s}(\R)$. Then, from \eqref{condighallapla} and a change of variable, we get 
\begin{align*}
PV \int_{\R}\{v_h(s)-v_h(\t )\}G(s-\t) d\t +\l v_h= f_h(s) \qquad\textrm{ for all $s\in \R$,}
\end{align*}
where $v_h=\frac{v(\cdot+h)-v}{|h|}$ and  $f_h=\frac{f(\cdot+h)-f}{|h|}$.  Applying  once again   \cite{fall2020regularity}[Theorem 1.4-$(iv)$]  and using the periodicity of $v$. we get 
$$
\|v_h\|_{C^{1,\s}(\R)}\leq   C (\|v_h\|_{L^\infty(\R)}+  \|f_h-\l v_h\|_{C^\s(\R)}).
$$
Letting $h\to 0$ and using \eqref{eq:C1-a-est}, we get $v\in C^{2,\s}(\R)$ in this case.\\
Finally, the regularity properties of $v$ may be combined with Fubini's theorem  to see that (\ref{condighallapla1}) holds pointwisely.
%
\QED
\end{proof}

\begin{Proposition}\label{prop:All-HYPcRANDAL}
 There exists a unique $\l_*>0$ such that the linear operator
 $L_*:= L_{\l_*}: X \to Y $  has the following properties.
\begin{itemize}
\item[(i)] The kernel $N(L_*)$ of $L_*$ is spanned by the function $\cos(\cdot).$
\item[(ii)] $L_*: X\cap W_1^\perp  \to Y \cap W_1^\perp$ is an isomorphism.
\end{itemize}
Moreover
\begin{equation}
  \label{eq:transversality-cond}
\partial_\l \Bigl|_{\l= \l_*} L_\l e_1  =V'(\l_*)e_1 \not  \in Y \cap W_1^\perp.
\end{equation}
 \end{Proposition}
 
\begin{proof} By Lemma~ \ref{derv222}, there exists a unique $\l_*>0$ such
that $V(\l_*)=0$ and $V'(\l_*)>0$.  This with \eqref{eq:-Lv-Fpurier-slab}  imply that
$N(L_*)=\textrm{span}\{e_1\} =W_1$ and we  get (i) and \eqref{eq:transversality-cond}.\\

We now prove (ii). 
We pick $h \in Y\cap W_1^\perp$ and consider the equation
\begin{equation}\label{invert}
L_*w =  h.
\end{equation}
Recalling \eqref{eq:L-lam-ek-eq-nuk-ek},  \eqref{invert} is uniquely solved  by the function 
$$w(s)=\sum_{\ell \in \N^{*}\setminus \{1\}}w_\ell e_\ell(s),$$ where 
\begin{equation}\label{eq:coeff}
w_\ell=\frac{1}{\pi V(\lambda_*\ell)}  \langle h, e_\ell  \rangle_{L^2_{p,e}}, \quad  \ell\ne 1, \quad \ell \ne 0.
\end{equation}
In addition, \begin{align}\label{eqchatacte2}
\sum_{\ell \in \N^{*}\setminus \{1\}} (1+\ell^2) \langle w, e_\ell
\rangle_{L^2_{p,e}} ^2 =\frac{1}{\pi^2 } \sum_{\ell \in \N^{*}\setminus \{1\}} \frac{(1+\ell^2)}{ V^2(\lambda_*\ell)}  \langle h, e_\ell
\rangle_{L^2_{p,e}} ^2.
\end{align}
Note that  $h\in H^{1}_{p,e}$ since  $h\in C^{1, \alpha}(\R) \subset H^{1}_{loc}(\R)$ and $g$ is periodic and even. This together with the asymptotic in Lemma \ref{asymtotic}(ii)  and the characterization  \eqref{eqchatacte} allow to see that the sum in \eqref{eqchatacte2} is finite. Hence $w\in H^{1}_{p,e} \subset    H^{1/2}_{p,e}$.   \\\\

To conclude the proof, we show  that $w\in C^{2, \a}(\R)$. 
By a  change of variable,  we have from \ref{lem:DERVATCOSNT}
$$
L_* w(s)=\frac{1}{ \l_*}  \int_{\R}\{w(s)-w(s-t )\}G(t/ \l_*) dt+\frac{1}{\l_*} F_{\lambda_*} \star w (s) -2 \pi w(s),
$$
where $F_{\lambda_*}(t)=F (t/\lambda_*)$ and $F$ is defined in \eqref{eq:def-F-}.
Consequently, \eqref{invert} reads 
\begin{align}\label{regulstart}
\int_{\R}(w(s)-w(\bar{s}))G((s-\bar{s})/ \l_*) d\bar{s}-2 \pi \lambda_* w(s)=f_*(s),
\end{align}
where $$f_*(s):=\lambda_*h(s)-F_{\lambda_*}\star w (s).$$
Now by Morrey's embedding  theorem  $w\in H^1_{p,e}\subset  C^{0,\frac{1}{2}-\e}(\R)$ for all $\e\in (0,1/2)$. 
Hence  $F_{\lambda_*}\star w \in C^{0,\frac{1}{2}-\e}(\R)$ (see  \eqref{eqrFin L1}). Thus $f_*\in   C^{0,\frac{1}{2}-\e}(\R).$ By \eqref{regulstart}  and  Lemma \ref{LemmTworegul}  we get  $w\in C^{1 }(\R)$ so that   $F_{\lambda_*}\star w \in C^{1}(\R)$.  Applying once again Lemma \ref{LemmTworegul} we obtain  $w\in C^{2 }(\R)$, yielding   $F_{\lambda_*}\star w \in C^{2}(\R)$. It then follows that $f^*\in C^{1, \a}(\R)$, so that  by Lemma \ref{LemmTworegul}  we get $w\in C^{2, \alpha}(\R).$ The proof is thus complete. 
\QED
\end{proof}

Combining Proposition~\ref{prop:All-HYPcRANDAL} with (\ref{eq:relationship-D-phi-L-lambda}) and the fact that
\begin{equation*}
\partial_\l \Bigl|_{\l= \l_*} D_\phi \Phi(\lambda,0)e_1  = -\frac{V'(\l_*)}{\lambda_*} e_1  \;  \not  \in \; Y \cap W_1^\perp,
\end{equation*}
we are now in position to apply the Crandall-Rabinowitz theorem \cite[Theorem 1.7]{M.CR}, which will give rise to the following
bifurcation property.

\begin{Theorem}\label{propPhi-crandall-rabinowitz}
\label{sec:nonl-probl-solve-1}
 Let $\l_*$ be defined as in Proposition~\ref{prop:All-HYPcRANDAL}, and consider $  X\cap X_1^\perp$ the   closed subspace of codimension one
such that $X=( X\cap X_1^\perp)\oplus X_1 $, and let
$\cD_\Phi \subset (0,+\infty) \times X$ be the open set defined
in (\ref{eq:DPhi}). Then there exists $b_0>0$ and a unique
$C^\infty$ curve
$$
(-b_0,b_0) \to \cD_\Phi, \qquad b \mapsto (\l(b), \phi_b)
$$
such that
\begin{itemize}
\item[(i)] $\Phi(\l(b), \phi_b)= \mathcal{N}(\l(b) + \phi_b)+\frac{1}{2} \equiv 0\quad$ for $b \in (-b_0,b_0)$,
\item[(ii)] $\l(0)= \l_*$,
\item[(iii)]  $\phi_b = b \bigl(\cos(\cdot) + v_b\bigr)$ for $b \in (-b_0,b_0)$, and
$$
(-b_0,b_0) \to X\cap X_1^\perp, \qquad b \mapsto v_b
$$
is a $C^\infty$ curve satisfying $v_0=0$.
\end{itemize}
\end{Theorem}

The proof of  Theorem~\ref{theo1} follows from  Theorem~\ref{propPhi-crandall-rabinowitz} by  considering for every   $k\geq 1$,  the operator $L_k:=L_{\frac{\l_*}{k}}$ as given in  \eqref{eq:def-L-lambda}  and  \eqref{eq:L-lam-ek-eq-nuk-ek}.

\appendix 
\section{Identities and estimates involving modified Bessel functions}\label{appendix}
This section collects useful properties of the families of modified  Bessel functions $I_{\nu}$ and $K_{\nu}$. We have used these properties in Section~\ref{sec: linearized}.

\subsection{General properties}\label{appendix1}

For $\nu \ge 0$, the function $I_\nu$ is defined on $[0,\infty)$ by  
 \begin{equation}\label{Bes1}
 I_{\nu}(x)=\sum^{\infty}_{m=0}\frac{(\frac{1}{2} x)^{\nu+2m}}{m!\Gamma(\nu+m+1)} \qquad \textrm{ for  $x \ge 0$}.
\end{equation}
We note that $I_\nu$ is smooth on $(0,\infty)$ and positive on $(0,\infty)$, while $I_0(0)=1$ for $I_\nu(0)=0$ for $\nu>0$.

Moreover, for $\nu \in \R$, the integral representation 
\begin{equation}\label{Bes2}
 K_{\nu}(x)= \int_{0}^{\infty} e^{-x \cosh (t)} \cosh (\nu t)  dt \qquad \textrm{ for  $x > 0$} 
\end{equation}
can be used to define the smooth positive function $K_\nu: (0,\infty) \to \R$.

\subsection{Derivatives}\label{appendix2}
For all  $x \in(0, + \infty)$, we have
\begin{align} 
K'_0(x)&=-K_1(x), \quad [xK_1(x)]'=-xK_0(x),\label{eq:formderibess}\\
xI'_{\nu}(x)&=xI_{\nu-1}(x)-\nu I_{\nu} (x),\label{relateI}\\
-x K'_{\nu}(x)&=x K_{\nu-1}(x)+\nu K_{\nu}(x), \label{relateK}
\end{align}
see e.g. \cite{El} and \cite{BariczPonnusamy}. 
\subsection{Asymptotic behaviour}\label{appendix2-asym}

Asymptotics of  $I_{\nu}$  and $K_{\nu}$ are given e.g. in \cite[Page 4]{HangYu}). In particular, we have 
 \begin{align} \label{eqasymp}
&\lim \limits_{x \to 0} x K_1(x)=1,\quad\quad \lim \limits_{x\to 0}  x K_0(x) =0\\
&\forall \mu\geq 0,  \quad K_\mu(x)\sim \frac{\sqrt{\pi}}{\sqrt{2}}x^{-\frac{1}{2}} e^{-x},\quad I_\mu(x)\sim \frac{1}{\sqrt{2\pi}}x^{-\frac{1}{2}} e^{x}, \quad \textrm{as}\quad x \rightarrow +\infty \label{eq:asyI1}\nonumber\\
\end{align}

\subsection{Inequalities}\label{appendix3}
Since $\cosh$ is a strictly increasing function on $[0,\infty)$, it follows from (\ref{Bes2}) that 
\begin{align} 
&K_\nu(x) < K_\mu(x) \qquad \text{for $0 \le \nu < \mu$ and $x \in (0,\infty)$} \label{eqCompK0K1}
 \end{align} 
In addition, \eqref{eqCompK0K1}, the first limit in \eqref{eqasymp} and the second identity in  \eqref{eq:formderibess} imply,
 \begin{align} \label{eqlimi11}
 x  K_0(x)< x  K_1(x)\leq 1 \quad \textrm{for all $ x \in(0, + \infty).$}
 \end{align}
The following inequalities for $x>0$ are also available, see e.g. \cite[(3.34), (3.35)]{Gaunt1}), \cite{HangYu}  and \cite{Nassel}:
\begin{align} \label{eq:compI1I00}
&\frac{x}{2+x}<  \frac{I_1(x)}{I_0(x)}<\frac{2x}{1+2x}\quad \textrm{and}\quad \frac{x}{2+x}<  \frac{I_1(x)}{I_0(x)}<\frac{x}{2},
\end{align}
\begin{align} \label{eq:compK1K0}
 & \frac{3+4x}{1+4x}<  \frac{K_1(x)}{K_0(x)}<\frac{1+2x}{2x},
\end{align}
\begin{align} \label{IneI}
 & I_{\nu}(x)> I_{\mu}(x), \quad \mu>\nu\geq 0,
\end{align}
\begin{align} 
& 0\leq x K_{\nu}(x) I_{\nu}(x)<\frac{1}{2};\quad \nu>\frac{1}{2} \label{eqINEK1}\\
&  \frac{1}{2}< x K_{\nu+1}(x) I_{\nu}(x)\leq 1; \quad \nu>-\frac{1}{2}. \label{eqINEK}
\end{align}

\subsection{A Wronskian identity}\label{appendix4}
For all $x\in (0, +\infty)$ we have (see e.g. \cite[page 79]{NevWatson})
\begin{align} \label{eq:valK2}
x^2K_2(x)=x^2K_0(x)+2 x K_1 (x), 
\end{align}
and  the following Wronskian identity (see e.g. \cite[page 251]{Olver}):
\begin{align} \label{relawrosk}
& I_{\nu}(x)K_{\nu+1}+I_{\nu+1}(x)K_{\nu}(x)=\frac{1}{x}.
\end{align}

\subsection{Some integral identities}\label{appendix5}
As noted in \cite[Page 442, $5.^7$]{I.S.I.M}, we have
\begin{equation}
  \label{eq:integral-identity-1}
  \int_0^\infty \frac{t \sin(at)}{(t^2+\beta^2)^{\nu}}\,dt = \frac{\Gamma(1-\nu)}{\sqrt{\pi}}\beta \Bigl(\frac{2\beta}{a}\Bigr)^{\frac{1}{2}-\nu}K_{\frac{3}{2}-\nu}(a \beta)=  \frac{\sqrt{\pi}}{\Gamma(\nu)} \beta \Bigl(\frac{2\beta}{a}\Bigr)^{\frac{1}{2}-\nu}K_{\frac{3}{2}-\nu}(a \beta)
\end{equation}
for $a, \beta>0$, $\nu >\frac{1}{2}$, so in particular 
\begin{equation}
  \label{eq:integral-identity-1-special-case}
  \int_0^\infty \frac{t \sin(at)}{(t^2+\beta^2)^{\frac{3}{2}}}\,dt = a K_{0}(a \beta)\quad \text{for $a, \beta>0$.}
\end{equation}
Similarly, we have, as noted in \cite[Page 442, $2.$]{I.S.I.M},
\begin{equation}
  \label{eq:integral-identity-2}
  \int_0^\infty \frac{\cos (at)}{(t^2+\beta^2)^{\nu}}\,dt =
\frac{\Gamma(1-\nu)}{\sqrt{\pi}} \Bigl(\frac{2\beta}{a}\Bigr)^{\frac{1}{2}-\nu}K_{\nu-\frac{1}{2}}(a \beta)=  \frac{\sqrt{\pi}}{\Gamma(\nu)}\Bigl(\frac{2\beta}{a}\Bigr)^{\frac{1}{2}-\nu}K_{\nu-\frac{1}{2}}(a \beta)
\end{equation}
for $a, \beta>0$, $\nu >0$, so in particular
\begin{equation}
  \label{eq:integral-identity-2-special-case}
  \int_0^\infty \frac{\cos (at)}{(t^2+\beta^2)^{\frac{3}{2}}}\,dt =\frac{a}{\beta} K_{1}(a \beta)\quad \text{for $a, \beta>0$.}
\end{equation}
and
\begin{equation}
  \label{eq:integral-identity-2-special-case-5-2}
  \int_0^\infty \frac{\cos (at)}{(t^2+\beta^2)^{\frac{5}{2}}}\,dt =\frac{a^2}{3 \beta^2} K_{2}(a \beta)\quad \text{for $a, \beta>0$.}
\end{equation}
We also recall the following identity, as stated in \cite[$13^8.$, Page 724]{I.S.I.M}:
\begin{equation}
  \label{eq:final-integral-identity}
\int_0^{\frac{\pi}{2}}K_{\nu-m}(2\rho \cos \theta)\cos((m+\nu) \theta)d \theta = (-1)^{m}\frac{\pi}{2}I_{m}(\rho)K_\nu(\rho)
\end{equation}
for every nonnegative integer $m$ and $\nu < m+1$.

\end{document}